\documentclass[12pt]{amsart}

\usepackage{amssymb}

\usepackage{enumerate}

\usepackage{graphicx}

\usepackage{amsmath}
\usepackage{amsfonts}
\usepackage{amsmath}
\usepackage{amsthm}
\usepackage{amssymb}
\usepackage{amscd}
\usepackage{mathrsfs}
\usepackage{latexsym}
\usepackage{times}
\usepackage{fancyhdr}
\usepackage{epsfig}
\usepackage[all]{xy}
\usepackage{xypic}
\usepackage[english]{babel}
\usepackage{tkz-graph}

\usepackage{times}
\usepackage{fancyhdr}
\usepackage{epsfig}
\usepackage[all]{xy}
\usepackage[english]{babel}
\usepackage{leftidx}
\usepackage{enumerate}
\usepackage{graphicx}
\usepackage{epigraph}

\usepackage{xypic}
\usepackage[english]{babel}
\usepackage{tkz-graph}
\usepackage{longtable}
\usepackage{supertabular}

\usepackage{helvet}
\usepackage{textcomp}

\usepackage{float}
\usepackage{longtable,booktabs}

\usepackage{pifont}
\usepackage{tikz}

\usepackage{tikz}
\usepackage{etoolbox}
\usepackage{enumitem}
\usepackage{enumitem}
\usepackage{color}

\usepackage{calc}
\usepackage{tikz}
\usetikzlibrary{decorations.markings}
\tikzstyle{vertex}=[circle, draw, inner sep=0pt, minimum size=6pt]
\newcommand{\vertex}{\node[vertex]}

\makeatletter
\@namedef{subjclassname@2010}{%
  \textup{2010} Mathematics Subject Classification}
\makeatother

\newtheorem{thm}{Theorem}[section]

\newtheorem{lem}[thm]{Lemma}
\newtheorem{prop}[thm]{Proposition}

\theoremstyle{definition}

\numberwithin{equation}{section}

\def\Gal{\text{Gal}}

\def\Im{\text{Im}}

\def\Gon{\text{Gon}}

\newcommand{\hooklongrightarrow}{\lhook\joinrel\longrightarrow}

\begin{document}


\baselineskip=17pt



\title[Kenku's method]{A note on the Kenku's method.}

\author[J. Wang]{Jian Wang}

\address{College of Mathematics\\ Jilin Normal University\\
Siping, Jilin 136000, China}
\email{blandye@gmail.com}

\date{\today}

\begin{abstract}
In this note, we present a complete and corrected version of Kenku's method in the proof of the nonexistence of 27-torsion on elliptic curves over quadratic fields and show why this method can not be applied to the cases of 32- and 24-torsion on elliptic curves over cubic fields.
\end{abstract}



\maketitle


In 1983, Kenku proved the nonexistence of 27-torsion on elliptic curves over quadratic fields \cite{Kenku83}. He considered the reduction modulo 5 of the corresponding modular curve $X_1(27)$ and reached the conclusion that a noncuspidal quadratic point on $X_1(27)$ will lead to contradiction.   Kenku's method is correct in practice but vaguely stated with errors and gaps. In this note, we present a complete and corrected version of Kenku's method and show why this method can not be applied to the cases of 32- and 24-torsion on elliptic curves over cubic fields.

\section{Preliminaries}

Let $\mathbb{H}=\{z\in\mathbb{C}~|~\Im z>0\}$ be the upper half plane. Let $\overline{\mathbb{H}}=\mathbb{H}\cup\mathbb{P}^1(\mathbb{Q})$ be the extended upper half plane by adjoining cusps $\mathbb{P}^1(\mathbb{Q})=\mathbb{Q}\cup\{\infty\}$ to $\mathbb{H}$. Let $N$ be a positive integer. Let $X_1(N)$ (resp. $X_0(N)$) be the modular curve defined over $\mathbb{Q}$ associated to the following congruence subgroup $\Gamma_1(N)$ (resp. $\Gamma_0(N)$) of the full modular group $\Gamma=SL_2(\mathbb{Z})/\pm1$.
$$\aligned\Gamma_0(N)&=\left\{\left(
                          \begin{array}{cc}
                            a & b \\
                            c & d \\
                          \end{array}
                        \right)\in SL_2(\mathbb{Z})/\pm1~~|~~c\equiv0\mod N
\right\}\\
\Gamma_1(N)&=\left\{\left(
                          \begin{array}{cc}
                            a & b \\
                            c & d \\
                          \end{array}
                        \right)\in \Gamma_0(N)~~|~~a\equiv d\equiv1\mod N
\right\}
\endaligned$$
The we have the Galois covering
$$\pi: X_1(N)\longrightarrow X_0(N)$$
with group $G=\Gamma_0(N)/\Gamma_1(N)=\left\{\left(
                                          \begin{array}{cc}
                                            a & 0 \\
                                            0 & a^{-1} \\
                                          \end{array}
                                        \right):a\in(\mathbb{Z}/N\mathbb{Z})^\times/\pm1\right\}$.
Let $\Delta$ be a subgroup of $(\mathbb{Z}/N\mathbb{Z})^\times$ which contains $\pm1$. Let $X_\Delta(N)$ be the modular curve defined over $\mathbb{Q}$ associated the congruence subgroup
$$\Gamma_\Delta(N)=\left\{\left(
                          \begin{array}{cc}
                            a & b \\
                            c & d \\
                          \end{array}
                        \right)\in \Gamma_0(N)~~|~~a\in\Delta\mod N
\right\}$$
Then all the intermediate modular curves between $X_1(N)$ and $X_0(N)$ are of the form $X_\Delta(N)$.

We denote by $Y_1(N)=X_1(N)\backslash\{cusps\},
Y_0(N)=X_0(N)\backslash\{cusps\},
Y_\Delta(N)=X_\Delta(N)\backslash\{cusps\}$
the corresponding affine curves. Denote by $J_1(N)$ (respectively, $J_0(N)$, $J_\Delta(N)$) the jacobian of $X_1(N)$ (respectively, $X_0(N)$, $X_\Delta(N)$).

The representation and behavior under the covering $\pi: X_1(N)\longrightarrow X_\Delta(N)\longrightarrow X_0(N)$ of the cusps was explained by Ogg \cite{Ogg73}. The cusps can be regarded as pairs $\pm$$x\choose y$, where $x,y$ are in $\mathbb{Z}/N\mathbb{Z}$, and relatively prime, and $x\choose y$, $-x\choose-y$ are identified. The cusps are rational in $\mathbb{Q}(\zeta_N)=\mathbb{Q}(e^{2\pi i/N})$ and cusps with the same $y$ are conjugate. A cusp on $X_1(N)$ is an orbit $\{\pm$$x+by\choose y$$\}$, $b\in\mathbb{Z}/N\mathbb{Z}$. If $N\geq5$, then for each $d|N$, there are $\frac12\varphi(d)\varphi(N/d)$ cusps of $X_1(N)$ corresponding to the $\varphi(d)$ $x$'s with $(x,d)=1$ and the $\varphi(N/d)$ $y$'s with $d=(y,N)$. The cusps of $X_\Delta(N)$ (and $X_0(N)$) are the orbits of $G$ on the cusps of $X_1(N)$ with the action
$$\left(
    \begin{array}{cc}
      a & 0 \\
      0 & a^{-1} \\
    \end{array}
  \right)\left(
           \begin{array}{c}
             x \\
             y \\
           \end{array}
         \right)=\left(
                          \begin{array}{c}
                            ax \\
                            a^{-1}y \\
                          \end{array}
                        \right)
$$
The cusps (especially of $X_0(N)$) at $d=1$ are also called infinity cusps (denoted as $\infty$) and cusps at $d=N$ called zero cusps (denoted as $0$).

For each divisor $N'$ of $N$ such that $(N',N/N')=1$, the matrix
$$\left(
    \begin{array}{cc}
      N'a & b \\
      Nc & N'd \\
    \end{array}
  \right)
$$
with $a,b,c,d\in\mathbb{Z}$ and determinant $N'$ defines a Atkin-Lehner involution $\omega_{N'}$ on $X_0(N)$. As noted in \cite{JeonKim}, it also defines an automorphism, not an involution in general, of $X_1(N)$ because it belongs to the normalizer of $\Gamma_1(N)$ in $PSL_2(\mathbb{R})$ \cite{KimKoo}. We also denote this automorphism as $\omega_{N'}$. The Atkin-Lehner involutions are mutually commutative and satisfies the multiplication rule
$$\omega_{N_1}\circ\omega_{N_2}=\omega_{lcm(N_1,N_2)/gcd(N_1,N_2)}$$
The action of $\omega_{N'}$ on the cusps of $X_0(N)$ was given by Ogg \cite[Proposition 2]{Ogg74}. It is easy to see $$\omega_{N'}\left(
               \begin{array}{c}
                 1 \\
                 1 \\
               \end{array}
             \right)=\left(
                       \begin{array}{c}
                         1 \\
                         N' \\
                       \end{array}
                     \right)
$$

The moduli interpretation of noncuspidal points of $X_1(N)$ are $(E,\pm P)$, where $E$ is an elliptic curve and $P\in E$ is a point of order $N$.  The moduli interpretation of noncuspidal points of $X_0(N)$ are $(E, C)$, where $E$ is an elliptic curve and $C\subset E$ is a cyclic subgroup of order $N$. The map $\pi: X_1(N)\longrightarrow X_0(N)$ send $(E,\pm P)$ to $(E,\langle P\rangle)$, where $\langle P\rangle$ is the cyclic subgroup generated by $P$.

Let $p$ be a prime such that $p\nmid N$. Igusa's theorem \cite{Igusa} says that the modular curves $X_1(N)$ and $X_0(N)$ have good reduction at prime $p$. Moreover, reducing the modular curve is compatible with reducing the moduli interpretation (See for example \cite[Theorem 1]{Ogg75}). And the description of the cusps is the same in characteristic $p$ as in characteristic $0$.

Let $K$ be a number field with ring of integers $\mathcal{O}_K$, $\wp\subset\mathcal{O}_K$ a prime ideal lying above $p$, $k=\mathbb{F}_q=\mathcal{O}_K/\wp$ its residue field. Let $E$ be an elliptic curve over $K$ and $P\in E(K)$ a point of order $N$. Let $\widetilde{E}$ be the fibre over $k$ of the N\'{e}ron model of $E$, and let $\widetilde{P}\in\widetilde{E}(k)$ be the reduction of $P$. Suppose that $p\nmid N$. Then elementary theory of group schemes shows that $\widetilde{P}$ has order $N$ due to the following well-known result (See for example \cite[\S 7.3 Proposition 3]{BoschLutkebohmertRaynaud}).

\begin{prop}\label{torsionreduction} Let $m$ be a positive integer relatively prime to $char(k)$. Then the reduction map
$$E(K)[m]\longrightarrow\widetilde{E}(k)$$
is injective.
\end{prop}

Let $k=\mathbb{F}_q$ be the finite field with $q=p^n$ elements. Let $E/k$ be an elliptic curve over $k$. Let $|E(k)|$ be the number of points of $E$ over $k$. Then Hasse's theorem states that
$$||E(k)|-q-1|\leq2\sqrt{q}$$
i.e.
$$(1-\sqrt{p^n})^2\leq|E(k)|\leq(1+\sqrt{p^n})^2$$
Let $t=q+1-|E(k)|$, $E$ is called ordinary if $(t,q)=1$, otherwise it is called supersingular. In the range proposed by Hasse's theorem, all the ordinary $t$ appear, while the supersingular $t$ only appears in restricted case. This was determined by Waterhouse \cite[Theorem 4.1]{Waterhouse}.

For the following two cases which we need in Section \ref{proof}, we list here the representatives of the $k$-isomorphism classes of elliptic curves over $k$.

Case 1:  $p\neq2$ or $3$, $j\neq0,1728$, by Silverman \cite[\S X Proposition 5.4]{Silverman},
we have the following representatives of the two isomorphism classes of elliptic curves over $k$ with $j(E_1)=j(E_2)=j$
$$\aligned &E_1: y^2=x^3-\frac{27j}{j-1728}x+\frac{54j}{j-1728}\\&E_2: y^2=x^3-\frac{27j}{j-1728}\alpha^2x+\frac{54j}{j-1728}\alpha^3\endaligned$$
where $\sqrt{\alpha}\not\in k^\times$. For $k=\mathbb{F}_{5^3}$, we take $\alpha=2$.

Case 2: $p=3$, $j\neq0$,  by Jeong \cite[Theorem 3.6]{Jeong}, we have the following representatives of the two isomorphism classes of elliptic curves over $k$ with $j(E_1)=j(E_2)=j$
$$\aligned &E_1: y^2=x^3+x^2-1/j\\&E_2: y^2=x^3+\alpha x^2-\alpha^3/j\endaligned$$
where $\sqrt{\alpha}\not\in k^\times$. For $k=\mathbb{F}_{3^3}$, we take $\alpha=2$.


\section{Kenku's Method}

Kenku's method was based on the following Lemma \ref{reductionlemma} and Lemma \ref{AtkinLehner}.

\begin{lem}\cite{Wang} \label{reductionlemma}
Suppose$N>4$ such that  $\Gon(X_1(N))>d$, $J_1(N)(\mathbb{Q})$ is finite, $p>2$ is a prime not dividing $N$. Let $K$ be a number field of degree $d$ over $\mathbb{Q}$ and $\wp$ a prime of $K$ over $p$. Let $E_{/K}$ be an elliptic curve with a $K$-rational point $P$ of order $N$, i.e. $x=(E,\pm P)\in Y_1(N)(K)$. Then $E$ has good reduction at $\wp$.
\end{lem}

The following Proposition is used in the proofs of Lemma \ref{AtkinLehner} and Lemma \ref{Frobenius} without declaration.

\begin{prop}\cite[Theorem 8.5.7]{DiamondShurman} Let $C$ and $C'$ be nonsingular projective algebraic curves over $\mathbb{Q}$ with good reduction at $p$ and let $C'$ have positive genus. For any morphism $h:C\longrightarrow C'$ the following diagram is commutative:
$$\xymatrix {C\ar[d]_{\otimes\overline{\mathbb{F}}_p}\ar[r]^{h}&C'\ar[d]^{\otimes\overline{\mathbb{F}}_p}\\
\widetilde{C}\ar[r]^{\widetilde{h}}&\widetilde{C}'}
$$
\end{prop}

\begin{lem}\label{AtkinLehner} Suppose $g(X_0(N))>0$. Let $p$ be a prime not dividing $N$. Let $K$ be a number field of degree $d$ over $\mathbb{Q}$ and $\wp$ a prime of $K$ over $p$. Let $E_{/K}$ be an elliptic curve with a $K$-rational point $P$ of order $N$, i.e. $x=(E,\pm P)\in Y_1(N)(K)$. Suppose $E$ has good reduction at $\wp$. Then $\omega_{N'}(\pi(x)\otimes\overline{\mathbb{F}}_p)$ also corresponds to an elliptic curve with a point of order $N$.
\end{lem}

\begin{proof} By the definitions of Atkin-Lehner involution $\omega_{N'}$ on $X_0(N)$ and the automorphism $\omega_{N'}$ on $X_1(N)$, the following diagram is commutative
$$\xymatrix {Y_1(N)\ar[d]^{\omega_{N'}}\ar[r]^{\pi}&Y_0(N)\ar[d]^{\omega_{N'}}\ar[r]^{\otimes\overline{\mathbb{F}}_p}&\widetilde{Y}_0(N)\ar[d]^{\omega_{N'}}\\
Y_1(N)\ar[r]^{\pi}&Y_0(N)\ar[r]^{\otimes\overline{\mathbb{F}}_p}&\widetilde{Y}_0(N)\\
}
$$
The following diagram is also commutative
$$\xymatrix {Y_1(N)\ar[d]^{\otimes\overline{\mathbb{F}}_p}\ar[r]^{\pi}&Y_0(N)\ar[d]^{\otimes\overline{\mathbb{F}}_p}\\
\widetilde{Y}_1(N)\ar[r]^{\pi}&\widetilde{Y}_0(N)}
$$
So
$$\omega_{N'}(\pi(x)\otimes\overline{\mathbb{F}}_p)=\pi(\omega_{N'}(x))\otimes\overline{\mathbb{F}}_p
=\pi(\omega_{N'}(x)\otimes\overline{\mathbb{F}}_p)$$
Since $x$ has good reduction, then $\omega_{N'}(x)$ also has good reduction. Now that $\omega_{N'}(x)\otimes\overline{\mathbb{F}}_p$ is the reduction of the point $\omega_{N'}(x)$ on $X_1(N)$, then it corresponds to an elliptic curve with a point of order $N$. Therefore $\omega_{N'}(\pi(x)\otimes\overline{\mathbb{F}}_p)=\pi(\omega_{N'}(x)\otimes\overline{\mathbb{F}}_p)$ also corresponds to an elliptic curve with a point of order $N$.

\end{proof}

\begin{lem}\label{Frobenius} Suppose $g(X_0(N))>0$. Let $p$ be a prime not dividing $N$. Let $K$ be a cubic number field and $\wp$ a prime of $K$ over $p$ with residue field $k=\mathbb{F}_q=\mathbb{F}_{p^3}$. Let $E_{/K}$ be an elliptic curve with a $K$-rational point $P$ of order $N$, i.e. $x=(E,\pm P)\in Y_1(N)(K)$. Suppose $E$ has good reduction at $\wp$. If the Galois covering $\pi:X_1(N)\longrightarrow X_0(N)$ factors as
$$X_1(N)\xrightarrow{~~~\pi_1~~~}X_\Delta(N)\xrightarrow{~~~\pi_2~~~}X_0(N)$$
with $g(X_\Delta(N))=g(X_0(N))=1$. Then on $\widetilde{X}_0(N)$, $$\pi(x)\otimes\overline{\mathbb{F}}_p+\phi(\pi(x)\otimes\overline{\mathbb{F}}_p)
+\phi^2(\pi(x)\otimes\overline{\mathbb{F}}_p)=\pi_2(z)\otimes\overline{\mathbb{F}}_p$$
for some rational point $z$ on $X_\Delta(N)$, where $$\aligned\phi:\widetilde{X}_0(N)&\longrightarrow \widetilde{X}_0(N)\\(x,y)&\longmapsto(x^p,y^p)\endaligned$$
is induced by the Frobenius map $\phi:\overline{\mathbb{F}}_p\longrightarrow\overline{\mathbb{F}}_p, x\longmapsto x^p$. Moreover, for any Atkin-Lehner $\omega_{N'}$, if $\omega_{N'}(x)$ is cubic, then
$$\omega_{N'}(\pi(x)\otimes\overline{\mathbb{F}}_p)+\phi(\omega_{N'}(\pi(x)\otimes\overline{\mathbb{F}}_p))
+\phi^2(\omega_{N'}(\pi(x)\otimes\overline{\mathbb{F}}_p))=\pi_2(z_{N'})\otimes\overline{\mathbb{F}}_p$$
for some rational point $z_{N'}$ on $X_\Delta(N)$.
\end{lem}

\begin{proof} Consider the Galois closure $L$ of $K$. Then either $\Gal(L/\mathbb{Q})=\mathbb{Z}/3\mathbb{Z}$ or $\Gal(L/\mathbb{Q})=S_3$. We claim that there is an element $\sigma\in\Gal(L/\mathbb{Q})$ of order $3$, such that (after a necessary rearrangement) the embeddings $\tau_i:K\hooklongrightarrow\mathbb{C}, 1\leq i\leq 3$ satisfy
$$\tau_i=\sigma^i|_K$$

In fact, if $\Gal(L/\mathbb{Q})=\mathbb{Z}/3\mathbb{Z}$, then $L=K$, i.e. $K/\mathbb{Q}$ is a Galois extension. So $\Gal(K/\mathbb{Q})=\{\tau_1,\tau_2,\tau_3\}$. Let $\sigma$ be a generator of $\Gal(K/\mathbb{Q})=\mathbb{Z}/3\mathbb{Z}$. Then, after a necessary rearrangement of $\tau_1,\tau_2$ and $\tau_3$, we have $\tau_i=\sigma^i$.

Otherwise, if $\Gal(L/\mathbb{Q})=S_3$, then $L/K$ is a quadratic extension. There is an element $\sigma_2\in\Gal(L/\mathbb{Q})$ of order $2$ such that $\Gal(L/K)=\langle\sigma_2\rangle$. Let $\sigma_3\in\Gal(L/\mathbb{Q})$ be an element of order $3$. Then $\Gal(L/\mathbb{Q})=\langle\sigma_2,\sigma_3\rangle$. On the other hand, each $\tau_i$ extends to two embedding $\tau_{i1},\tau_{i2}:L\hooklongrightarrow\mathbb{Q}$ and $\Gal(L/\mathbb{Q})=\{\tau_{11},\tau_{12},\tau_{21},\tau_{22},\tau_{31},\tau_{32}\}$. Without loss of generality, suppose $\tau_3$ is the identity embedding, then $\{\tau_{31},\tau_{32}\}=\{id,\sigma_2\}$, and after a necessary rearrangement of $\tau_1$ and $\tau_2$, $\{\tau_{11},\tau_{12}\}=\{\sigma_3,\sigma_3\sigma_2\}, \{\tau_{21},\tau_{22}\}=\{\sigma_3^2,\sigma_3^2\sigma_2\}$. Let $\sigma=\sigma_3$. Then $\tau_i=\sigma^i|_K$.

The projection of the rational divisor class $[x_1+x_2+x_3-d\infty]\in J_1(N)$ on $J_\Delta(N)$ is also rational. i.e.
$$\aligned \pi_{1*}(\left[\sigma(x)+\sigma^2(x)+\sigma^3(x)-d\infty\right])&=\pi_{1*}([\tau_1(x)+\tau_2(x)+\tau_3(x)-d\infty])\\&=\pi_{1*}([x_1+x_2+x_3-d\infty])\in J_\Delta(N)(\mathbb{Q})
\endaligned$$
Since the genus of $X_\Delta(N)$ is $1$, i.e. $X_\Delta(N)$ is an elliptic curve,  then there is a canonical isomorphism
$$\aligned \iota: J_\Delta(N)&\longrightarrow X_\Delta(N)\\\left[\sum_i n_ix_i\right]&\longmapsto\sum_i n_ix_i\endaligned$$
where the right $\sum_i$ is the group law on the elliptic curve $X_\Delta(N)$.
Since the Galois action is compatible with the projection, i.e. the following diagram is commutive
$$\xymatrix{J_1(N)\ar[d]^{\sigma}\ar[r]^{\pi_{1*}}&J_\Delta(N)\ar[d]^{\sigma}\\
J_1(N)\ar[r]^{\pi_{1*}}&J_\Delta(N)}
$$
then
$$\aligned&\iota\circ\pi_{1*}(\left[\sigma(x)+\sigma^2(x)+\sigma^3(x)-d\infty\right])\\=&\iota(\left[\sigma(\pi_1(x))+\sigma^2(\pi_1(x))+\sigma^3(\pi_1(x))-d\infty\right])
\\=&\sigma(\pi_1(x))+\sigma^2(\pi_1(x))+\sigma^3(\pi_1(x))\in X_\Delta(N)(\mathbb{Q})
\endaligned$$
i.e. there is a rational point $z$ on $X_\Delta(N)$ such that
$$\sigma(\pi_1(x))+\sigma^2(\pi_1(x))+\sigma^3(\pi_1(x))=z$$

Let $\wp'$ be a prime of $L$ over $\wp$ with residue field $k'=\mathcal{O}_L/\wp'$ (which is an extension of $k$). It is known in algebraic number theory that the Frobenius $\phi\in\Gal(k'/\mathbb{F}_p)$ is the reduction from a Frobenius element $Frob_{\wp'}$ in $\Gal(L/\mathbb{Q})$. It is easy to see $k'=k$ since the highest order of an element in $\Gal(L/\mathbb{Q})$ is $3$. Now, we have an element $\sigma'(:=Frob_{\wp'})\in \Gal(L/\mathbb{Q})$ of order $3$ such that the following diagram is commutative:
$$\xymatrix{X_\Delta(N)\ar[r]^{\pi_2}\ar[d]^{\sigma'}&X_0(N)\ar[r]^{\otimes\overline{\mathbb{F}}_p}\ar[d]^{\sigma'}&\widetilde{X}_0(N)\ar[d]^{\phi}\\
X_\Delta(N)\ar[r]^{\pi_2}&X_0(N)\ar[r]^{\otimes\overline{\mathbb{F}}_p}&\widetilde{X}_0(N)}
$$
It is easy to know either $\sigma'=\sigma$ or $\sigma'=\sigma^2$. In both cases, we always have
$$id+\sigma'+\sigma'^2=\sigma+\sigma^2+\sigma^3$$
Since both $X_\Delta(N)$ and $X_0(N)$ are elliptic curves, they are equipped with the group law. We can always choose an infinity cusp $\infty_\Delta$ on $X_\Delta(N)$ as the zero point $O_\Delta$ and the unique infinity cusp $\infty_1$ on $X_0(N)$ as the zero point $O_1$. Since the projection $\pi_2:X_\Delta(N)\longrightarrow X_0(N)$ maps $\infty_\Delta$ to $\infty_1$, then it is an isogeny, thus preserves the group law. And our assumption $p\nmid N$ guarantees $X_0(N)$ have good reduction at $\wp$. Then the group law is also preserved during reduction. In other words, the following diagram is commutative:
$$\xymatrix{X_\Delta(N)\ar[r]^{\pi_2}\ar[d]^{+}&X_0(N)\ar[r]^{\otimes\overline{\mathbb{F}}_p}\ar[d]^{+}&\widetilde{X}_0(N)\ar[d]^{+}\\
X_\Delta(N)\ar[r]^{\pi_2}&X_0(N)\ar[r]^{\otimes\overline{\mathbb{F}}_p}&\widetilde{X}_0(N)}
$$
Therefore
$$\aligned&\pi(x)\otimes\overline{\mathbb{F}}_p+\phi(\pi(x)\otimes\overline{\mathbb{F}}_p)
+\phi^2(\pi(x)\otimes\overline{\mathbb{F}}_p)\\
=&(\pi_2\circ\pi_1(x))\otimes\overline{\mathbb{F}}_p+\phi((\pi_2\circ\pi_1(x))\otimes\overline{\mathbb{F}}_p)
+\phi^2((\pi_2\circ\pi_1(x))\otimes\overline{\mathbb{F}}_p)\\
=&\pi_2(\pi_1(x))\otimes\overline{\mathbb{F}}_p+\pi_2(\sigma'(\pi_1(x)))\otimes\overline{\mathbb{F}}_p
+\pi_2(\sigma'^2(\pi_1(x)))\otimes\overline{\mathbb{F}}_p\\
=&(\pi_2(\pi_1(x))+\pi_2(\sigma'(\pi_1(x)))
+\pi_2(\sigma'^2(\pi_1(x))))\otimes\overline{\mathbb{F}}_p\\
=&\pi_2(\pi_1(x)+\sigma'(\pi_1(x))
+\sigma'^2(\pi_1(x)))\otimes\overline{\mathbb{F}}_p\\
=&\pi_2(\sigma(\pi_1(x))+\sigma^2(\pi_1(x))
+\sigma^3(\pi_1(x)))\otimes\overline{\mathbb{F}}_p\\
=&\pi_2(z)\otimes\overline{\mathbb{F}}_p\\
\endaligned$$
as desired.

Moreover, for any Atkin-Lehner involution $\omega_{N'}$,  we have
$$\omega_{N'}(\pi(x)\otimes\overline{\mathbb{F}}_p)=\pi(\omega_{N'}(x))\otimes\overline{\mathbb{F}}_p$$
because of the following commutative diagram:$$\xymatrix {Y_1(N)\ar[d]^{\omega_{N'}}\ar[r]^{\pi}&Y_0(N)\ar[d]^{\omega_{N'}}\ar[r]^{\otimes\overline{\mathbb{F}}_p}&\widetilde{Y}_0(N)\ar[d]^{\omega_{N'}}\\
Y_1(N)\ar[r]^{\pi}&Y_0(N)\ar[r]^{\otimes\overline{\mathbb{F}}_p}&\widetilde{Y}_0(N)\\
}
$$
If $\omega_{N'}(x)$ is cubic, then we all the arguments that work for $x$ also work for $\omega_{N'}(x)$.
Therefore
$$\aligned&\omega_{N'}(\pi(x)\otimes\overline{\mathbb{F}}_p)+\phi(\omega_{N'}(\pi(x)\otimes\overline{\mathbb{F}}_p))
+\phi^2(\omega_{N'}(\pi(x)\otimes\overline{\mathbb{F}}_p))\\
=&\pi(\omega_{N'}(x))\otimes\overline{\mathbb{F}}_p+\phi(\pi(\omega_{N'}(x))\otimes\overline{\mathbb{F}}_p)
+\phi^2(\pi(\omega_{N'}(x))\otimes\overline{\mathbb{F}}_p)\\
=&\pi_2(z_{N'})\otimes\overline{\mathbb{F}}_p\endaligned$$
for some rational point $z_{N'}$ on $X_\Delta(N)$.
\end{proof}

\section{The 32- and 24-torsion of elliptic curves over cubic fields}\label{proof}

While the Atkin-Lehner involution $\omega_N$ on $X_0(N)$ is defined over $\mathbb{Q}$, the Atkin-Lehner involution $\omega_N$ on $X_1(N)$ is actually defined over $\mathbb{Q}(\zeta_N)$ rather than $\mathbb{Q}$ as we wish. In this section, we \texttt{PRETEND} that the Atkin-Lehner involution $\omega_N$ on $X_1(N)$ is defined over $\mathbb{Q}$. Under this assumption, for all $N'$ with $(N',N/N')=1$, $\omega_{N'}(x)$ is cubic if $x$ is cubic (as required in the last statement of Lemma \ref{Frobenius}). With this \texttt{FALSE} assumption, the Kenku's method is used to "\texttt{PROVE}" the nonexistence of 32- and 24-torsion on elliptic curves over cubic fields.

\subsection{$N=32$}

As is shown in \cite{Wang}, $J_1(N)(\mathbb{Q})$ is finite. From Jeon-Kim-Schweizer \cite{JeonKimSchweizer}, we know $Gon(X_1(N))>3$. Let $K$ be a cubic field and $\wp$ a prime of $K$ over $3$. Suppose $x=(E,\pm P)\in Y_1(N)(K)$. Therefore by Lemma \ref{reductionlemma}, $E$ has good reduction at $\wp$. By Proposition \ref{torsionreduction}, the reduction $\widetilde{P}$ of $P$ is a $k$-rational point of order $N$ in the elliptic curve $\widetilde{E}$ over $k=\mathcal{O}_K/\wp$.

The following lemma shows that $\widetilde{E}(k)$ can not have a point of order $32$. This contradiction implies that $\mathbb{Z}/32\mathbb{Z}$ is not a subgroup of $E(K)_{tor}$. The calculations on the finite field $\mathbb{F}_{3^3}$ in this Lemma is done in Sage \cite{Sage}.

\begin{lem}
$\widetilde{E}(k)$ can not have a point of order $32$
\end{lem}

\begin{proof}

For the modular curve $X_0(32)$,
Fricke \cite{Fricke} calculated the defining equation and the explicit formula of $j$-invariant and the Atkin-Lehner $\omega_{32}$ on this equation. These data can also be found in Furumoto-Hasegawa \cite{FurumotoHasegawa}. The equation is
\begin{equation}\label{equation32-2} y^2=x^3+6x^2+16x+16
\end{equation}
with the formula of $\omega_{32}$
$$\omega_{32}(P)=-P\pm(0,4)$$
and the formula of $j$-invariant
$$j=\frac{256((x(x+4)/2)^4+8(x(x+4)/2)^3+20(x(x+4)/2)^2+16(x(x+4)/2)+1)^3}{(x(x+4)/2)((x(x+4)/2)+4)((x(x+4)/2)+2)^2}$$
It can be calculated in Sage \cite{Sage} that the set of rational torsion points of this equation is
$$\aligned X_0(32)(\mathbb{Q})_{tor}&=\{(\infty,\infty), (0, 4), (-2, 0), (0, -4)\} &\cong\mathbb{Z}/4\mathbb{Z}\endaligned$$
This is exactly all the rational points of this equation since $X_0(32)(\mathbb{Q})\cong J_0(32)(\mathbb{Q})$ is finite.

The map $\pi:X_1(32)\longrightarrow X_0(32)$ factors through $X_\Delta(32)$, where $\Delta=\{\pm1,\pm7,\pm9,\pm15\}$, with $g(X_\Delta(32))=1$ (See for example \cite{JeonKim07}). The cusps of $X_\Delta(32)$ over the $0$ cusp of $X_0(32)$ are quadratic. This is because the eight cusps of $X_1(32)$ over $0$ are
$$\pm\left(
    \begin{array}{c}
      1 \\
      0 \\
    \end{array}
  \right),\pm\left(
    \begin{array}{c}
      3 \\
      0 \\
    \end{array}
  \right),\pm\left(
    \begin{array}{c}
      5 \\
      0 \\
    \end{array}
  \right),\pm\left(
    \begin{array}{c}
      7 \\
      0 \\
    \end{array}
  \right),\pm\left(
    \begin{array}{c}
      9 \\
      0 \\
    \end{array}
  \right),\pm\left(
    \begin{array}{c}
      11 \\
      0 \\
    \end{array}
  \right),\pm\left(
    \begin{array}{c}
      13 \\
      0 \\
    \end{array}
  \right),\pm\left(
    \begin{array}{c}
      15 \\
      0 \\
    \end{array}
  \right)
$$
They have two orbits under the action of $\Delta$:
$$\aligned&\left\{\pm\left(
    \begin{array}{c}
      1 \\
      0 \\
    \end{array}
  \right),\pm\left(
    \begin{array}{c}
      7 \\
      0 \\
    \end{array}
  \right),\pm\left(
    \begin{array}{c}
      9 \\
      0 \\
    \end{array}
  \right),\pm\left(
    \begin{array}{c}
      15 \\
      0 \\
    \end{array}
  \right)\right\}\\
  &\left\{\pm\left(
    \begin{array}{c}
      3 \\
      0 \\
    \end{array}
  \right),\pm\left(
    \begin{array}{c}
      5 \\
      0 \\
    \end{array}
  \right),\pm\left(
    \begin{array}{c}
      11 \\
      0 \\
    \end{array}
  \right),\pm\left(
    \begin{array}{c}
      13 \\
      0 \\
    \end{array}
  \right)\right\}\endaligned
$$
So $X_\Delta(32)$ has two cusps $\pm$$1\choose0$, $\pm$$3\choose0$ over $0$. They are conjugates since they have the same $y$ in the representation $\pm$$x\choose y$.

Now if $\omega_{32}(P)=-P+(0,4)$, then $(0,4)=\omega_{32}(\infty)=0$, which means $(0,4)$ comes from quadratic cusps of $X_0(32)$ under $\pi_2:X_\Delta(32)\longrightarrow X_0(32)$. If $\omega_{32}(P)=-P-(0,4)=-P+(0,-4)$, then $(0,-4)=\omega_{32}(\infty)=0$, which means $(0,-4)$ comes from quadratic cusps of $X_0(32)$ under $\pi_2:X_\Delta(32)\longrightarrow X_0(32)$.

If $k=\mathbb{F}_3$ or $\mathbb{F}_{3^2}$, then $\widetilde{E}(k)$ can not have a point of order $32$ since $32>(1+\sqrt{3^2})^2$. If $k=\mathbb{F}_{3^3}$, let $[m,n,l]$ denote the element $ma^2+na+l\in\mathbb{F}_{3^3}=\mathbb{F}_3[a]$. Table \ref{table32-27point} list the coordinates and $j$-invariant of the points on the reduction of the equation \ref{equation32-2} modulo 3 which are rational over $\mathbb{F}_{3^3}$. Denote $\varphi=id+\phi+\phi^2$. The last two columns shows $\varphi(X,Y)$.

\begin{table}[!ht]
\tabcolsep 0pt
\vspace*{0pt}
\begin{center}
\def\temptablewidth{0.8\textwidth}
\setlength{\abovecaptionskip}{0pt}
\setlength{\belowcaptionskip}{-5pt}
\caption{Points on $\widetilde{X}_0(32)(\mathbb{F}_{3^3})$}
\label{table32-27point}
{\rule{\temptablewidth}{1pt}}
\begin{tabular*}{\temptablewidth}{@{\extracolsep{\fill}}ccccccccccccc}
$X$&$Y_1$&$Y_2$&$j$&$\varphi(X,Y_1)$&$\varphi(X,Y_2)$\\\hline
$\infty$&$\infty$&$\infty$&$-$\\
$0$&$2$&$1$&$-$\\
$1$&$0$&$0$&$-$\\
$[0,1,0]$&$[2,1,1]$&$[1,2,2]$&$[2,1,2](*)$&$(0,2)$&$(0,1)$\\
$[0,1,1]$&$[2,2,1]$&$[1,1,2]$&$[2,2,2](*)$&$(0,2)$&$(0,1)$\\
$[0,1,2]$&$[2,0,2]$&$[1,0,1]$&$[2,0,0](*)$&$(0,2)$&$(0,1)$\\
$[1,0,1]$&$[1,1,0]$&$[2,2,0]$&$[2,1,2](*)$&$(1,0)$&$(1,0)$\\
$[1,1,2]$&$[1,2,0]$&$[2,1,0]$&$[2,0,0](*)$&$(1,0)$&$(1,0)$\\
$[1,2,2]$&$[1,0,2]$&$[2,0,1]$&$[2,2,2](*)$&$(1,0)$&$(1,0)$\\
$[2,0,0]$&$[1,0,0]$&$[2,0,0]$&$[2,2,1]$&$(\infty,\infty)$&$(\infty,\infty)$\\
$[2,0,2]$&$[0,2,2]$&$[0,1,1]$&$[2,1,1]$&$(0,2)$&$(0,1)$\\
$[2,1,1]$&$[0,2,1]$&$[0,1,2]$&$[2,2,1]$&$(0,2)$&$(0,1)$\\
$[2,1,2]$&$[1,2,1]$&$[2,1,2]$&$[2,0,2]$&$(\infty,\infty)$&$(\infty,\infty)$\\
$[2,2,1]$&$[0,2,0]$&$[0,1,0]$&$[2,0,2]$&$(0,2)$&$(0,1)$\\
$[2,2,2]$&$[1,1,1]$&$[2,2,2]$&$[2,1,1]$&$(\infty,\infty)$&$(\infty,\infty)$\\
\end{tabular*}
{\rule{\temptablewidth}{1pt}}
\end{center}
\end{table}

Table \ref{table32-27group} shows the group structures of elliptic curves over $\mathbb{F}_{3^3}$ with the given $j$-invariants that appear in Table \ref{table32-27point}. It is easy to see that only the point $P$ with $j$-invariants marked by $(*)$ may correspond to an elliptic curve with a point of order $32$.

\begin{table}[!ht]
\tabcolsep 0pt
\vspace*{0pt}
\begin{center}
\def\temptablewidth{1\textwidth}
\setlength{\abovecaptionskip}{0pt}
\setlength{\belowcaptionskip}{-5pt}
\caption{Group structures of elliptic curves over $\mathbb{F}_{3^3}$}
\label{table32-27group}
{\rule{\temptablewidth}{1pt}}
\begin{tabular*}{\temptablewidth}{@{\extracolsep{\fill}}ccccccccccccc}
$j$&$E_1(\mathbb{F}_{3^3})$&$E_2(\mathbb{F}_{3^3})$&$j$&$E_1(\mathbb{F}_{3^3})$&$E_2(\mathbb{F}_{3^3})$\\\hline
$[2,0,0](*)$&$\mathbb{Z}/24$&$\mathbb{Z}/32$&
$[2,0,2]$&$\mathbb{Z}/2\times\mathbb{Z}/12$&$\mathbb{Z}/2\times\mathbb{Z}/16$\\
$[2,1,1]$&$\mathbb{Z}/2\times\mathbb{Z}/12$&$\mathbb{Z}/2\times\mathbb{Z}/16$&
$[2,1,2](*)$&$\mathbb{Z}/24$&$\mathbb{Z}/32$\\
$[2,2,1]$&$\mathbb{Z}/2\times\mathbb{Z}/12$&$\mathbb{Z}/2\times\mathbb{Z}/16$&
$[2,2,2](*)$&$\mathbb{Z}/24$&$\mathbb{Z}/32$\\
\end{tabular*}
{\rule{\temptablewidth}{1pt}}
\end{center}
\end{table}

Denote $\omega(P)=-P+(0,4)\equiv-P+(0,1)\mod3$ and $\omega'(P)=-P-(0,4)=-P+(0,-4)\equiv-P+(0,2)\mod3$. By using the additive law on $X_0(32)$, both $\omega(P)$ and $\omega'(P)$ are calculated as listed in Table \ref{table32-27involution}.

\begin{table}[!ht]
\tabcolsep 0pt
\vspace*{0pt}
\begin{center}
\def\temptablewidth{0.8\textwidth}
\setlength{\abovecaptionskip}{0pt}
\setlength{\belowcaptionskip}{-5pt}
\caption{Atkin-Lehner involution on $\widetilde{X}_0(32)(\mathbb{F}_{3^3})$}
\label{table32-27involution}
{\rule{\temptablewidth}{1pt}}
\begin{tabular*}{\temptablewidth}{@{\extracolsep{\fill}}ccccccccccccc}
$~~~~P~~~~$&$X(P)$&$Y(P)$&$X(\omega(P))$&$Y(\omega(P))$&$X(\omega'(P))$&$Y(\omega'(P))$\\\hline
\texttt{1}&$[0,1,0]$&$[2,1,1]$&$[2,2,2]$&$[2,2,2]$&$[1,0,1]$&$[1,1,0]$\\
\texttt{2}&$[0,1,0]$&$[1,2,2]$&$[1,0,1]$&$[2,2,0]$&$[2,2,2]$&$[1,1,1]$\\
\texttt{3}&$[0,1,1]$&$[2,2,1]$&$[2,0,0]$&$[2,0,0]$&$[1,2,2]$&$[1,0,2]$\\
\texttt{4}&$[0,1,1]$&$[1,1,2]$&$[1,2,2]$&$[2,0,1]$&$[2,0,0]$&$[1,0,0]$\\
\texttt{5}&$[0,1,2]$&$[2,0,2]$&$[2,1,2]$&$[2,1,2]$&$[1,1,2]$&$[1,2,0]$\\
\texttt{6}&$[0,1,2]$&$[1,0,1]$&$[1,1,2]$&$[2,1,0]$&$[2,1,2]$&$[1,2,1]$\\
\texttt{7}&$[1,0,1]$&$[1,1,0]$&$[2,2,1]$&$[0,1,0]$&$[0,1,0]$&$[2,1,1]$\\
\texttt{8}&$[1,0,1]$&$[2,2,0]$&$[0,1,0]$&$[1,2,2]$&$[2,2,1]$&$[0,2,0]$\\
\texttt{9}&$[1,1,2]$&$[1,2,0]$&$[2,1,1]$&$[0,1,2]$&$[0,1,2]$&$[2,0,2]$\\
\texttt{10}&$[1,1,2]$&$[2,1,0]$&$[0,1,2]$&$[1,0,1]$&$[2,1,1]$&$[0,2,1]$\\
\texttt{11}&$[1,2,2]$&$[1,0,2]$&$[2,0,2]$&$[0,1,1]$&$[0,1,1]$&$[2,2,1]$\\
\texttt{12}&$[1,2,2]$&$[2,0,1]$&$[0,1,1]$&$[1,1,2]$&$[2,0,2]$&$[0,2,2]$\\
\texttt{13}&$[2,0,0]$&$[1,0,0]$&$[2,0,2]$&$[0,2,2]$&$[0,1,1]$&$[1,1,2]$\\
\texttt{14}&$[2,0,0]$&$[2,0,0]$&$[0,1,1]$&$[2,2,1]$&$[2,0,2]$&$[0,1,1]$\\
\texttt{15}&$[2,0,2]$&$[0,2,2]$&$[2,0,0]$&$[1,0,0]$&$[1,2,2]$&$[2,0,1]$\\
\texttt{16}&$[2,0,2]$&$[0,1,1]$&$[1,2,2]$&$[1,0,2]$&$[2,0,0]$&$[2,0,0]$\\
\texttt{17}&$[2,1,1]$&$[0,2,1]$&$[2,1,2]$&$[1,2,1]$&$[1,1,2]$&$[2,1,0]$\\
\texttt{18}&$[2,1,1]$&$[0,1,2]$&$[1,1,2]$&$[1,2,0]$&$[2,1,2]$&$[2,1,2]$\\
\texttt{19}&$[2,1,2]$&$[1,2,1]$&$[2,1,1]$&$[0,2,1]$&$[0,1,2]$&$[1,0,1]$\\
\texttt{20}&$[2,1,2]$&$[2,1,2]$&$[0,1,2]$&$[2,0,2]$&$[2,1,1]$&$[0,1,2]$\\
\texttt{21}&$[2,2,1]$&$[0,2,0]$&$[2,2,2]$&$[1,1,1]$&$[1,0,1]$&$[2,2,0]$\\
\texttt{22}&$[2,2,1]$&$[0,1,0]$&$[1,0,1]$&$[1,1,0]$&$[2,2,2]$&$[2,2,2]$\\
\texttt{23}&$[2,2,2]$&$[1,1,1]$&$[2,2,1]$&$[0,2,0]$&$[0,1,0]$&$[1,2,2]$\\
\texttt{24}&$[2,2,2]$&$[2,2,2]$&$[0,1,0]$&$[2,1,1]$&$[2,2,1]$&$[0,1,0]$\\

\end{tabular*}
{\rule{\temptablewidth}{1pt}}
\end{center}
\end{table}

Consider the points $P=(X,Y)$ numbered by $\texttt{1},\cdots,\texttt{12}$ with $j(P)$ among those marked by $(*)$. Now, if $\omega_{32}=\omega$, then only the even numbered points \texttt{2},\texttt{4},\texttt{6},\texttt{8},\texttt{10},\texttt{12} satisfy that $j(\omega_{32}(P))$ is among those marked by $(*)$ in Table \ref{table32-27group}, which is required by Lemma \ref{AtkinLehner}. But they have $$\aligned\varphi(P)&=(0,1)&(\texttt{2},\texttt{4},\texttt{6})\\\varphi(\omega_{32}(P))&=(0,1)&~~~~~~~(\texttt{8},\texttt{10},\texttt{12})\endaligned$$ which is not allowed by Lemma \ref{Frobenius}.

Otherwise, if $\omega_{32}=\omega'$, then only the odd numbered points \texttt{1},\texttt{3},\texttt{5},\texttt{7},\texttt{9},\texttt{11} satisfy that $j(\omega_{32}(P))$ is among those marked by $(*)$ in Table \ref{table32-27group}, which is required by Lemma \ref{AtkinLehner}. But they have $$\aligned\varphi(P)&=(0,2)&(\texttt{1},\texttt{3},\texttt{5})\\\varphi(\omega_{32}(P))&=(0,2)&~~~~~~~(\texttt{7},\texttt{9},\texttt{11})\endaligned$$ which is not allowed  by Lemma \ref{Frobenius}.

Actually, the information in Table \ref{table32-27involution} can be represented in Figure \ref{graph32mod27a} and Figure \ref{graph32mod27b}. The black dot means either the point has unmarked $j$-invariant or the point has $\varphi(P)$ not allowed by Lemma \ref{Frobenius}. The white dot point is satisfactory, but each of them is connected to a black dot point by $\omega_{32}$, which is not allowed by Lemma \ref{AtkinLehner} or Lemma \ref{Frobenius}.

\begin{figure}
\caption{Atkin-Lehner involution on $\widetilde{X}_0(32)(\mathbb{F}_{3^3})$, $\omega_{32}=\omega$}
\[\begin{tikzpicture}\label{graph32mod27a}
	\vertex (1) at (3*0.9914,3*0.1305) [label=right:$\texttt{1}$] {};
	\vertex[fill] (2) at (3*0.9239,3*0.3827) [label=right:$\texttt{2}$] {};
	\vertex (3) at (3*0.7934,3*0.6088) [label=right:$\texttt{3}$] {};
	\vertex[fill] (4) at (3*0.6088,3*0.7934) [label=right:$\texttt{4}$] {};
	\vertex (5) at (3*0.3827,3*0.9239) [label=right:$\texttt{5}$] {};
	\vertex[fill] (6) at (3*0.1305,3*0.9914) [label=right:$\texttt{6}$] {};
	\vertex (7) at (-3*0.1305,3*0.9914) [label=left:$\texttt{7}$] {};
	\vertex (8) at (-3*0.3827,3*0.9239) [label=left:$\texttt{8}$] {};
	\vertex (9) at (-3*0.6088,3*0.7934) [label=left:$\texttt{9}$] {};
    \vertex (10) at (-3*0.7934,3*0.6088) [label=left:$\texttt{10}$] {};
	\vertex (11) at (-3*0.9239,3*0.3827) [label=left:$\texttt{11}$] {};
	\vertex (12) at (-3*0.9914,3*0.1305) [label=left:$\texttt{12}$] {};
    \vertex[fill] (13) at (-3*0.9914,-3*0.1305) [label=left:$\texttt{13}$] {};
	\vertex[fill] (14) at (-3*0.9239,-3*0.3827) [label=left:$\texttt{14}$] {};
	\vertex[fill] (15) at (-3*0.7934,-3*0.6088) [label=left:$\texttt{15}$] {};
    \vertex[fill] (16) at (-3*0.6088,-3*0.7934) [label=left:$\texttt{16}$] {};
	\vertex[fill] (17) at (-3*0.3827,-3*0.9239) [label=left:$\texttt{17}$] {};
	\vertex[fill] (18) at (-3*0.1305,-3*0.9914) [label=left:$\texttt{18}$] {};
    \vertex[fill] (19) at (3*0.1305,-3*0.9914) [label=right:$\texttt{19}$] {};
	\vertex[fill] (20) at (3*0.3827,-3*0.9239) [label=right:$\texttt{20}$] {};
	\vertex[fill] (21) at (3*0.6088,-3*0.7934) [label=right:$\texttt{21}$] {};
	\vertex[fill] (22) at (3*0.7934,-3*0.6088) [label=right:$\texttt{22}$] {};
	\vertex[fill] (23) at (3*0.9239,-3*0.3827) [label=right:$\texttt{23}$] {};
	\vertex[fill] (24) at (3*0.9914,-3*0.1305) [label=right:$\texttt{24}$] {};

    	\path
		(1) edge (24)
        (2) edge (8)
        (3) edge (14)
        (4) edge (12)
        (5) edge (20)
        (6) edge (10)
        (7) edge (22)
        (8) edge (2)
        (9) edge (18)
        (10) edge (6)
        (11) edge (16)
        (12) edge (4)
        (13) edge (15)
        (14) edge (3)
        (15) edge (13)
        (16) edge (11)
        (17) edge (19)
        (18) edge (9)
        (19) edge (17)
        (20) edge (5)
        (21) edge (23)
        (22) edge (7)
        (23) edge (21)
        (24) edge (1)

	;
\end{tikzpicture}\]
\end{figure}

\begin{figure}
\caption{Atkin-Lehner involution on $\widetilde{X}_0(32)(\mathbb{F}_{3^3})$, $\omega_{32}=\omega'$}
\[\begin{tikzpicture}\label{graph32mod27b}
	\vertex[fill] (1) at (3*0.9914,3*0.1305) [label=right:$\texttt{1}$] {};
	\vertex (2) at (3*0.9239,3*0.3827) [label=right:$\texttt{2}$] {};
	\vertex[fill] (3) at (3*0.7934,3*0.6088) [label=right:$\texttt{3}$] {};
	\vertex (4) at (3*0.6088,3*0.7934) [label=right:$\texttt{4}$] {};
	\vertex[fill] (5) at (3*0.3827,3*0.9239) [label=right:$\texttt{5}$] {};
	\vertex (6) at (3*0.1305,3*0.9914) [label=right:$\texttt{6}$] {};
	\vertex (7) at (-3*0.1305,3*0.9914) [label=left:$\texttt{7}$] {};
	\vertex (8) at (-3*0.3827,3*0.9239) [label=left:$\texttt{8}$] {};
	\vertex (9) at (-3*0.6088,3*0.7934) [label=left:$\texttt{9}$] {};
    \vertex (10) at (-3*0.7934,3*0.6088) [label=left:$\texttt{10}$] {};
	\vertex (11) at (-3*0.9239,3*0.3827) [label=left:$\texttt{11}$] {};
	\vertex (12) at (-3*0.9914,3*0.1305) [label=left:$\texttt{12}$] {};
    \vertex[fill] (13) at (-3*0.9914,-3*0.1305) [label=left:$\texttt{13}$] {};
	\vertex[fill] (14) at (-3*0.9239,-3*0.3827) [label=left:$\texttt{14}$] {};
	\vertex[fill] (15) at (-3*0.7934,-3*0.6088) [label=left:$\texttt{15}$] {};
    \vertex[fill] (16) at (-3*0.6088,-3*0.7934) [label=left:$\texttt{16}$] {};
	\vertex[fill] (17) at (-3*0.3827,-3*0.9239) [label=left:$\texttt{17}$] {};
	\vertex[fill] (18) at (-3*0.1305,-3*0.9914) [label=left:$\texttt{18}$] {};
    \vertex[fill] (19) at (3*0.1305,-3*0.9914) [label=right:$\texttt{19}$] {};
	\vertex[fill] (20) at (3*0.3827,-3*0.9239) [label=right:$\texttt{20}$] {};
	\vertex[fill] (21) at (3*0.6088,-3*0.7934) [label=right:$\texttt{21}$] {};
	\vertex[fill] (22) at (3*0.7934,-3*0.6088) [label=right:$\texttt{22}$] {};
	\vertex[fill] (23) at (3*0.9239,-3*0.3827) [label=right:$\texttt{23}$] {};
	\vertex[fill] (24) at (3*0.9914,-3*0.1305) [label=right:$\texttt{24}$] {};

    	\path
		(1) edge (7)
        (2) edge (23)
        (3) edge (11)
        (4) edge (13)
        (5) edge (9)
        (6) edge (19)
        (7) edge (1)
        (8) edge (21)
        (9) edge (5)
        (10) edge (17)
        (11) edge (3)
        (12) edge (15)
        (13) edge (4)
        (14) edge (16)
        (15) edge (12)
        (16) edge (14)
        (17) edge (10)
        (18) edge (20)
        (19) edge (6)
        (20) edge (18)
        (21) edge (8)
        (22) edge (24)
        (23) edge (2)
        (24) edge (22)

	;
\end{tikzpicture}\]
\end{figure}

\end{proof}

\subsection{$N=24$}

As is shown in \cite{Wang}, $J_1(N)(\mathbb{Q})$ is finite. From Jeon-Kim-Schweizer \cite{JeonKimSchweizer}, we know $Gon(X_1(N))>3$. Let $K$ be a cubic field and $\wp$ a prime of $K$ over $5$. We can always choose $\wp$ such that the residue field $k=\mathcal{O}_K/\wp$ has degree $1$ or $3$ over $\mathbb{F}_p$. In fact, the decomposition of $p$ in $\mathcal{O}_K$ has the following five types
$$\aligned&I: &p\mathcal{O}_K&=\wp&&II: &p\mathcal{O}_K&=\wp^3& &III: &p\mathcal{O}_K&=\wp_1\wp_2&\\ &IV: &p\mathcal{O}_K&=\wp_1\wp_2^2&
 &V: &p\mathcal{O}_K&=\wp_1\wp_2\wp_3&\endaligned$$
In type $II,IV,V$, all the primes over $p$ have degree $1$ residue field. In type $I$, the prime over $p$ has degree $3$ residue field. In type $III$, the degree of the residue fields of the two primes $\wp_1,\wp_2$ is $1$ and $2$ respectively. We choose the one with degree $1$ residue field as $\wp$.

Suppose $x=(E,\pm P)\in Y_1(N)(K)$. Therefore by Lemma \ref{reductionlemma}, $E$ has good reduction at $\wp$. By Proposition \ref{torsionreduction}, the reduction $\widetilde{P}$ of $P$ is a $k$-rational point of order $N$ in the elliptic curve $\widetilde{E}$ over $k=\mathcal{O}_K/\wp$.

The following lemma shows that $\widetilde{E}(k)$ can not have a point of order $24$. This contradiction implies that $\mathbb{Z}/32\mathbb{Z}$ is not a subgroup of $E(K)_{tor}$. The calculations on the finite field $\mathbb{F}_{5^3}$ in this Lemma is done in Sage \cite{Sage}.

\begin{lem}
$\widetilde{E}(k)$ can not have a point of order $24$.
\end{lem}

\begin{proof}

For the modular curve $X_0(24)$,
Fricke \cite{Fricke} calculated the defining equation and the explicit formula of $j$-invariant and the Atkin-Lehner involutions $\omega_4,\omega_9,\omega_{36}$ on this equation. These data can also be found in Furumoto-Hasegawa \cite{FurumotoHasegawa}. The equation is
\begin{equation}\label{equation24-2} y^2=x^3+11x^2+36x+36
\end{equation}
with the formula of the Atkin-Lehner involutions
$$\aligned\omega_{24}(P)&=-P+(0,6)\\\omega_3(P)&=P+(-3,0)\\\omega_8(P)&=\omega_{24}\circ\omega_3^{-1}(P)=-P+(-4,2)\endaligned$$
and the formula of $j$-invariant
$$j=\frac{27\left(\frac{(x(x+6)/2)(2(x(x+6)/2)+9)^2}{27((x(x+6)/2)+4)}+1\right)\left(9\frac{(x(x+6)/2)(2(x(x+6)/2)+9)^2}{27((x(x+6)/2)+4)}+1\right)^3}{\frac{(x(x+6)/2)(2(x(x+6)/2)+9)^2}{27((x(x+6)/2)+4)}}$$
It can be calculated in Sage \cite{Sage} that the set of rational torsion points of this equation is
$$\aligned X_0(24)(\mathbb{Q})&=\{(\infty,\infty), (-4, 2), (-2, 0), (-4, -2), (-3,0),(0,6),(-6,0),(0,-6)\}\\&\cong\mathbb{Z}/2\mathbb{Z}\times\mathbb{Z}/4\mathbb{Z}\endaligned$$
This is exactly all the rational points of this equation since $X_0(24)(\mathbb{Q})\cong J_0(24)(\mathbb{Q})$ is finite.

The map $\pi:X_1(24)\longrightarrow X_0(24)$ factors through $X_\Delta(24)$, with $\Delta=\{\pm1,\pm11\}$, $g(X_\Delta(24))=1$ (See for example \cite{JeonKim07}). The cusps of $X_\Delta(24)$ over the $0$ cusp of $X_0(24)$ are quadratic. This is because the four cusps of $X_1(24)$ over $0$ are
$$\pm\left(
    \begin{array}{c}
      1 \\
      0 \\
    \end{array}
  \right),\pm\left(
    \begin{array}{c}
      5 \\
      0 \\
    \end{array}
  \right),\pm\left(
    \begin{array}{c}
      7 \\
      0 \\
    \end{array}
  \right),\pm\left(
    \begin{array}{c}
      11 \\
      0 \\
    \end{array}
  \right)
$$
They have two orbits under the action of $\Delta$:
$$\left\{\pm\left(
    \begin{array}{c}
      1 \\
      0 \\
    \end{array}
  \right),\pm\left(
    \begin{array}{c}
      11 \\
      0 \\
    \end{array}
  \right)\right\},~~~~~
  \left\{\pm\left(
    \begin{array}{c}
      5 \\
      0 \\
    \end{array}
  \right),\pm\left(
    \begin{array}{c}
      7 \\
      0 \\
    \end{array}
  \right)\right\}
$$
So $X_\Delta(24)$ has two cusps $\pm$$1\choose0$, $\pm$$5\choose0$ over $0$. They are conjugates since they have the same $y$ in the representation $\pm$$x\choose y$.

The cusps of $X_\Delta(24)$ over the cusp $\pm$$1\choose3$ at $d=3$ of $X_0(24)$ are quadratic. This is because the four cusps of $X_1(24)$ over $\pm$$1\choose3$ are
$$\pm\left(
    \begin{array}{c}
      1 \\
      3 \\
    \end{array}
  \right),\pm\left(
    \begin{array}{c}
      2 \\
      3 \\
    \end{array}
  \right),\pm\left(
    \begin{array}{c}
      1 \\
      9 \\
    \end{array}
  \right),\pm\left(
    \begin{array}{c}
      2 \\
      9 \\
    \end{array}
  \right)
$$
They have two orbits under the action of $\Delta$:
$$\left\{\pm\left(
    \begin{array}{c}
      1 \\
      3 \\
    \end{array}
  \right),\pm\left(
    \begin{array}{c}
      2 \\
      9 \\
    \end{array}
  \right)\right\},~~~~~
  \left\{\pm\left(
    \begin{array}{c}
      2 \\
      3 \\
    \end{array}
  \right),\pm\left(
    \begin{array}{c}
      1 \\
      9 \\
    \end{array}
  \right)\right\}
$$
So $X_\Delta(24)$ has two cusps $\pm$$1\choose3$, $\pm$$2\choose3$ over $\pm$$1\choose3$. They are conjugates since they have the same $y$ in the representation $\pm$$x\choose y$.

The cusps of $X_\Delta(24)$ over the cusp $\pm$$1\choose8$ at $d=8$ of $X_0(24)$ are quadratic. This is because the four cusps of $X_1(24)$ over $\pm$$1\choose8$ are
$$\pm\left(
    \begin{array}{c}
      1 \\
      8 \\
    \end{array}
  \right),\pm\left(
    \begin{array}{c}
      3 \\
      8 \\
    \end{array}
  \right),\pm\left(
    \begin{array}{c}
      5 \\
      8 \\
    \end{array}
  \right),\pm\left(
    \begin{array}{c}
      7 \\
      8 \\
    \end{array}
  \right)
$$
They have two orbits under the action of $\Delta$:
$$\left\{\pm\left(
    \begin{array}{c}
      1 \\
      8 \\
    \end{array}
  \right),\pm\left(
    \begin{array}{c}
      5 \\
      8 \\
    \end{array}
  \right)\right\},~~~~~
  \left\{\pm\left(
    \begin{array}{c}
      3 \\
      8 \\
    \end{array}
  \right),\pm\left(
    \begin{array}{c}
      7 \\
      8 \\
    \end{array}
  \right)\right\}
$$
So $X_\Delta(24)$ has two cusps $\pm$$1\choose8$, $\pm$$3\choose8$ over $\pm$$1\choose8$. They are conjugates since they have the same $y$ in the representation $\pm$$x\choose y$.

Now  $(0,6)=\omega_{24}(\infty)=0$, $(-3,0)=\omega_3(\infty)=\pm$$1\choose3$, $(-4,2)=\omega_8(\infty)=\pm$$1\choose8$ which means $(0,6), (-3,0)$ and $(-4,2)$ all come from quadratic cusps of $X_0(24)$ under $\pi_2:X_\Delta(24)\longrightarrow X_0(24)$.

If $k=\mathbb{F}_5$, then $\widetilde{E}(k)$ can not have a point of order $24$ since $24>(1+\sqrt{5})^2$. If $k=\mathbb{F}_{5^3}$, let $[m,n,l]$ denote the element $ma^2+na+l\in\mathbb{F}_{5^3}=\mathbb{F}_5[a]$. Table \ref{table24-125point} list the coordinates, $j$-invariant and $\varphi$ of the  points on the reduction of equation \ref{equation24-2} modulo 5 which are rational over $\mathbb{F}_{5^3}$.

\begin{center}
\begin{longtable}{cccccccccccc}

\caption{Points on $\widetilde{X}_0(24)(\mathbb{F}_{5^3})$} \label{table24-125point} \\
\hline\hline \multicolumn{1}{c}{$X$} & \multicolumn{1}{c}{$Y_1$} & \multicolumn{1}{c}{$Y_2$} & \multicolumn{1}{c}{$j$}&\multicolumn{1}{c}{$\varphi(X,Y_1)$}&\multicolumn{1}{c}{$\varphi(X,Y_2)$}\\ \hline
\endfirsthead

\multicolumn{5}{c}%
{{ \tablename\ \thetable{}: continued}} \\
\hline \multicolumn{1}{c}{$X$} & \multicolumn{1}{c}{$Y_1$} & \multicolumn{1}{c}{$Y_2$} & \multicolumn{1}{c}{$j$}&\multicolumn{1}{c}{$\varphi(X,Y_1)$}&\multicolumn{1}{c}{$\varphi(X,Y_2)$}\\ \hline
\endhead

\hline
\endfoot

\hline\hline
\endlastfoot

$\infty$&$\infty$&$\infty$&$-$\\
$0$&$4$&$1$&$-$\\
$1$&$3$&$2$&$-$\\
$2$&$0$&$0$&$-$\\
$3$&$0$&$0$&$-$\\
$4$&$0$&$0$&$-$\\
$[0,1,1]$&$[3,0,1]$&$[2,0,4]$&$[4,4,4]$&$(0,1)$&$(0,4)$\\
$[0,1,3]$&$[4,1,4]$&$[1,4,1]$&$[4,3,0]$&$(1,2)$&$(1,3)$\\
$[0,2,1]$&$[3,2,4]$&$[2,3,1]$&$[1,0,4]$&$(4,0)$&$(4,0)$\\
$[0,3,0]$&$[4,1,2]$&$[1,4,3]$&$[4,1,0]$&$(2,0)$&$(2,0)$\\
$[0,4,0]$&$[1,0,2]$&$[4,0,3]$&$[2,1,3]$&$(1,2)$&$(1,3)$\\
$[0,4,3]$&$[2,3,2]$&$[3,2,3]$&$[4,4,4]$&$(0,1)$&$(0,4)$\\
$[1,0,3]$&$[3,4,4]$&$[2,1,1]$&$[4,3,0]$&$(4,0)$&$(4,0)$\\
$[1,1,1]$&$[1,3,4]$&$[4,2,1]$&$[3,4,2]$&$(1,2)$&$(1,3)$\\
$[1,1,2]$&$[4,0,4]$&$[1,0,1]$&$[3,4,4]$&$(4,0)$&$(4,0)$\\
$[1,1,3]$&$[2,4,4]$&$[3,1,1]$&$[4,4,2]$&$(0,1)$&$(0,4)$\\
$[1,2,2]$&$[2,2,3]$&$[3,3,2]$&$[2,2,1]$&$(2,0)$&$(2,0)$\\
$[1,3,3]$&$[3,4,1]$&$[2,1,4]$&$[4,4,2]$&$(\infty,\infty)$&$(\infty,\infty)$\\
$[1,3,4]$&$[0,1,0]$&$[0,4,0]$&$[2,1,2]$&$(3,0)$&$(3,0)$\\
$[1,4,3]$&$[3,2,1]$&$[2,3,4]$&$[3,4,0]$&$(\infty,\infty)$&$(\infty,\infty)$\\
$[1,4,4]$&$[2,4,4]$&$[3,1,1]$&$[2,3,2]$&$(3,0)$&$(3,0)$\\
$[2,0,2]$&$[2,1,2]$&$[3,4,3]$&$[3,2,2]$&$(0,1)$&$(0,4)$\\
$[2,0,4]$&$[4,3,3]$&$[1,2,2]$&$[1,1,1]$&$(1,2)$&$(1,3)$\\
$[2,1,0]$&$[3,0,1]$&$[2,0,4]$&$[3,4,0]$&$(0,1)$&$(0,4)$\\
$[2,1,3]$&$[1,4,4]$&$[4,1,1]$&$[3,2,2]$&$(1,2)$&$(1,3)$\\
$[2,1,4]$&$[0,3,1]$&$[0,2,4]$&$[3,2,4]$&$(4,0)$&$(4,0)$\\
$[2,2,3]$&$[1,0,2]$&$[4,0,3]$&$[4,1,0]$&$(3,0)$&$(3,0)$\\
$[2,2,4]$&$[3,3,1]$&$[2,2,4]$&$[4,4,1]$&$(\infty,\infty)$&$(\infty,\infty)$\\
$[2,3,0]$&$[0,1,0]$&$[0,4,0]$&$[3,2,0]$&$(0,1)$&$(0,4)$\\
$[2,3,3]$&$[3,3,3]$&$[2,2,2]$&$[4,4,4]$&$(1,2)$&$(1,3)$\\
$[2,3,4]$&$[1,2,3]$&$[4,3,2]$&$[4,4,1]$&$(4,0)$&$(4,0)$\\
$[2,4,0]$&$[0,1,0]$&$[0,4,0]$&$[3,4,2]$&$(0,1)$&$(0,4)$\\
$[2,4,2]$&$[2,2,0]$&$[3,3,0]$&$[1,0,4]$&$(1,2)$&$(1,3)$\\
$[3,0,2]$&$[2,4,4]$&$[3,1,1]$&$[3,2,2]$&$(0,1)$&$(0,4)$\\
$[3,0,4]$&$[4,2,4]$&$[1,3,1]$&$[0,2,2]$&$(1,2)$&$(1,3)$\\
$[3,1,1]$&$[0,2,0]$&$[0,3,0]$&$[2,3,3]$&$(1,2)$&$(1,3)$\\
$[3,1,4]$&$[1,1,0]$&$[4,4,0]$&$[3,4,2]$&$(0,1)$&$(0,4)$\\
$[3,2,1]$&$[0,2,0]$&$[0,3,0]$&$[4,4,4]$&$(1,2)$&$(1,3)$\\
$[3,2,2]$&$[2,4,1]$&$[3,1,4]$&$[1,1,1]$&$(2,0)$&$(2,0)$\\
$[3,2,3]$&$[4,4,4]$&$[1,1,1]$&$[2,1,2]$&$(0,1)$&$(0,4)$\\
$[3,3,2]$&$[4,4,3]$&$[1,1,2]$&$[3,2,0]$&$(\infty,\infty)$&$(\infty,\infty)$\\
$[3,3,3]$&$[3,0,1]$&$[2,0,4]$&$[1,1,0]$&$(3,0)$&$(3,0)$\\
$[3,4,1]$&$[1,0,2]$&$[4,0,3]$&$[3,2,2]$&$(1,2)$&$(1,3)$\\
$[3,4,2]$&$[0,1,2]$&$[0,4,3]$&$[2,3,3]$&$(2,0)$&$(2,0)$\\
$[3,4,3]$&$[3,2,2]$&$[2,3,3]$&$[1,1,0]$&$(0,1)$&$(0,4)$\\
$[4,0,3]$&$[4,2,2]$&$[1,3,3]$&$[4,2,0]$&$(2,0)$&$(2,0)$\\
$[4,1,2]$&$[4,3,3]$&$[1,2,2]$&$[4,2,0]$&$(3,0)$&$(3,0)$\\
$[4,1,3]$&$[1,4,2]$&$[4,1,3]$&$[3,2,4]$&$(\infty,\infty)$&$(\infty,\infty)$\\
$[4,2,2]$&$[0,2,0]$&$[0,3,0]$&$[2,2,1]$&$(3,0)$&$(3,0)$\\
$[4,2,3]$&$[1,3,2]$&$[4,2,3]$&$[3,4,4]$&$(\infty,\infty)$&$(\infty,\infty)$\\
$[4,3,4]$&$[4,4,1]$&$[1,1,4]$&$[0,2,2]$&$(4,0)$&$(4,0)$\\
$[4,4,0]$&$[3,4,2]$&$[2,1,3]$&$[2,3,2]$&$(0,1)$&$(0,4)$\\
$[4,4,3]$&$[4,3,3]$&$[1,2,2]$&$[3,4,2]$&$(1,2)$&$(1,3)$\\
$[4,4,4]$&$[3,0,3]$&$[2,0,2]$&$[2,1,3]$&$(2,0)$&$(2,0)$\\

\end{longtable}

\end{center}

\begin{table}[!ht]
\tabcolsep 0pt
\vspace*{0pt}
\begin{center}
\def\temptablewidth{1\textwidth}
\setlength{\abovecaptionskip}{0pt}
\setlength{\belowcaptionskip}{-5pt}
\caption{Group structures of the elliptic curves over $\mathbb{F}_{5^3}$}
\label{table24-125group}
{\rule{\temptablewidth}{1pt}}
\begin{tabular*}{\temptablewidth}{@{\extracolsep{\fill}}ccccccccccccc}
$j$&$E_1(\mathbb{F}_{5^3})$&$E_2(\mathbb{F}_{5^3})$&$j$&$E_1(\mathbb{F}_{5^3})$&$E_2(\mathbb{F}_{5^3})$\\\hline
$[0,2,2]$&$\mathbb{Z}/120$&$\mathbb{Z}/132$&
$[1,0,4]$&$\mathbb{Z}/120$&$\mathbb{Z}/132$\\
$[1,1,0]$&$\mathbb{Z}/132$&$\mathbb{Z}/120$&
$[1,1,1]$&$\mathbb{Z}/120$&$\mathbb{Z}/132$\\
$[2,1,2]$&$\mathbb{Z}/132$&$\mathbb{Z}/120$&
$[2,1,3]$&$\mathbb{Z}/120$&$\mathbb{Z}/132$\\
$[2,2,1]$&$\mathbb{Z}/144$&$\mathbb{Z}/108$&
$[2,3,2]$&$\mathbb{Z}/132$&$\mathbb{Z}/120$\\
$[2,3,3]$&$\mathbb{Z}/120$&$\mathbb{Z}/132$&
$[3,2,0]$&$\mathbb{Z}/120$&$\mathbb{Z}/132$\\
$[3,2,2]$&$\mathbb{Z}/2\times\mathbb{Z}/72$&$\mathbb{Z}/2\times\mathbb{Z}/54$&
$[3,2,4]$&$\mathbb{Z}/108$&$\mathbb{Z}/144$\\
$[3,4,0]$&$\mathbb{Z}/120$&$\mathbb{Z}/132$&
$[3,4,2]$&$\mathbb{Z}/2\times\mathbb{Z}/72$&$\mathbb{Z}/2\times\mathbb{Z}/54$\\
$[3,4,4]$&$\mathbb{Z}/108$&$\mathbb{Z}/144$&
$[4,1,0]$&$\mathbb{Z}/144$&$\mathbb{Z}/108$\\
$[4,2,0]$&$\mathbb{Z}/144$&$\mathbb{Z}/108$&
$[4,3,0]$&$\mathbb{Z}/120$&$\mathbb{Z}/132$\\
$[4,4,1]$&$\mathbb{Z}/108$&$\mathbb{Z}/144$&
$[4,4,2]$&$\mathbb{Z}/120$&$\mathbb{Z}/132$\\
$[4,4,4]$&$\mathbb{Z}/2\times\mathbb{Z}/72$&$\mathbb{Z}/2\times\mathbb{Z}/54$\\
\end{tabular*}
{\rule{\temptablewidth}{1pt}}
\end{center}
\end{table}

Table \ref{table24-125group} shows the group structures of elliptic curves over $\mathbb{F}_{5^3}$ with the given $j$-invariants that appear in Table \ref{table24-125point}. It can be seen that all noncuspidal pionts on $\widetilde{X}_0(24)(\mathbb{F}_{5^3})$ may correspond to an elliptic curve with a point of order $24$. By using the additive law on $X_0(24)$, $\omega_3(P),\omega_8(P)$ and $\omega_{24}(P)$ are calculated as listed in Table \ref{table24-125involution}.

\begin{center}
\begin{longtable}{cccccccccccc}

\caption{Atkin-Lehner involutions on $\widetilde{X}_0(24)(\mathbb{F}_{5^3})$} \label{table24-125involution} \\
\hline\hline \multicolumn{1}{c}{$P$} & \multicolumn{1}{c}{$X(P)$} & \multicolumn{1}{c}{$Y(P)$} & \multicolumn{1}{c}{$X(\omega_3(P))$} & \multicolumn{1}{c}{$Y(\omega_3(P))$} & \multicolumn{1}{c}{$X(\omega_8(P))$} & \multicolumn{1}{c}{$Y(\omega_8(P))$} & \multicolumn{1}{c}{$X(\omega_{24}(P))$} & \multicolumn{1}{c}{$Y(\omega_{24}(P))$}\\ \hline
\endfirsthead

\multicolumn{9}{c}%
{{ \tablename\ \thetable{}: continued}} \\
\hline \multicolumn{1}{c}{$P$} & \multicolumn{1}{c}{$X(P)$} & \multicolumn{1}{c}{$Y(P)$} & \multicolumn{1}{c}{$X(\omega_3(P))$} & \multicolumn{1}{c}{$Y(\omega_3(P))$} & \multicolumn{1}{c}{$X(\omega_8(P))$} & \multicolumn{1}{c}{$Y(\omega_8(P))$} & \multicolumn{1}{c}{$X(\omega_{24}(P))$} & \multicolumn{1}{c}{$Y(\omega_{24}(P))$}\\ \hline
\endhead

\hline
\endfoot

\hline\hline
\endlastfoot

\texttt{1}&$(0,1,1)$&$(3,0,1)$&$(4,4,3)$&$(4,3,3)$&$(2,1,4)$&$(0,2,4)$&$(4,1,2)$&$(1,2,2)$\\
\texttt{2}&$(0,1,1)$&$(2,0,4)$&$(4,4,3)$&$(1,2,2)$&$(1,2,2)$&$(2,2,3)$&$(2,2,4)$&$(3,3,1)$\\
\texttt{3}&$(0,1,3)$&$(4,1,4)$&$(2,3,0)$&$(0,1,0)$&$(1,3,4)$&$(0,4,0)$&$(0,2,1)$&$(3,2,4)$\\
\texttt{4}&$(0,1,3)$&$(1,4,1)$&$(2,3,0)$&$(0,4,0)$&$(1,4,3)$&$(3,2,1)$&$(4,4,4)$&$(2,0,2)$\\
\texttt{5}&$(0,2,1)$&$(3,2,4)$&$(1,3,4)$&$(0,4,0)$&$(2,3,0)$&$(0,1,0)$&$(0,1,3)$&$(4,1,4)$\\
\texttt{6}&$(0,2,1)$&$(2,3,1)$&$(1,3,4)$&$(0,1,0)$&$(4,4,0)$&$(3,4,2)$&$(2,0,4)$&$(4,3,3)$\\
\texttt{7}&$(0,3,0)$&$(4,1,2)$&$(4,1,3)$&$(1,4,2)$&$(3,0,2)$&$(3,1,1)$&$(3,2,1)$&$(0,3,0)$\\
\texttt{8}&$(0,3,0)$&$(1,4,3)$&$(4,1,3)$&$(4,1,3)$&$(0,4,3)$&$(3,2,3)$&$(1,1,1)$&$(4,2,1)$\\
\texttt{9}&$(0,4,0)$&$(1,0,2)$&$(3,4,3)$&$(3,2,2)$&$(1,4,4)$&$(2,4,4)$&$(4,3,4)$&$(1,1,4)$\\
\texttt{10}&$(0,4,0)$&$(4,0,3)$&$(3,4,3)$&$(2,3,3)$&$(3,3,2)$&$(1,1,2)$&$(3,4,2)$&$(0,1,2)$\\
\texttt{11}&$(0,4,3)$&$(2,3,2)$&$(1,1,1)$&$(1,3,4)$&$(1,1,2)$&$(4,0,4)$&$(4,2,2)$&$(0,2,0)$\\
\texttt{12}&$(0,4,3)$&$(3,2,3)$&$(1,1,1)$&$(4,2,1)$&$(0,3,0)$&$(1,4,3)$&$(4,1,3)$&$(4,1,3)$\\
\texttt{13}&$(1,0,3)$&$(3,4,4)$&$(3,3,3)$&$(2,0,4)$&$(2,1,0)$&$(3,0,1)$&$(3,0,4)$&$(4,2,4)$\\
\texttt{14}&$(1,0,3)$&$(2,1,1)$&$(3,3,3)$&$(3,0,1)$&$(3,2,3)$&$(4,4,4)$&$(3,1,1)$&$(0,2,0)$\\
\texttt{15}&$(1,1,1)$&$(1,3,4)$&$(0,4,3)$&$(2,3,2)$&$(4,2,2)$&$(0,2,0)$&$(1,1,2)$&$(4,0,4)$\\
\texttt{16}&$(1,1,1)$&$(4,2,1)$&$(0,4,3)$&$(3,2,3)$&$(4,1,3)$&$(4,1,3)$&$(0,3,0)$&$(1,4,3)$\\
\texttt{17}&$(1,1,2)$&$(4,0,4)$&$(4,2,2)$&$(0,2,0)$&$(0,4,3)$&$(2,3,2)$&$(1,1,1)$&$(1,3,4)$\\
\texttt{18}&$(1,1,2)$&$(1,0,1)$&$(4,2,2)$&$(0,3,0)$&$(3,0,2)$&$(2,4,4)$&$(3,2,1)$&$(0,2,0)$\\
\texttt{19}&$(1,1,3)$&$(2,4,4)$&$(2,4,2)$&$(2,2,0)$&$(4,3,4)$&$(4,4,1)$&$(1,4,4)$&$(3,1,1)$\\
\texttt{20}&$(1,1,3)$&$(3,1,1)$&$(2,4,2)$&$(3,3,0)$&$(3,4,2)$&$(0,4,3)$&$(3,3,2)$&$(4,4,3)$\\
\texttt{21}&$(1,2,2)$&$(2,2,3)$&$(2,2,4)$&$(3,3,1)$&$(0,1,1)$&$(2,0,4)$&$(4,4,3)$&$(1,2,2)$\\
\texttt{22}&$(1,2,2)$&$(3,3,2)$&$(2,2,4)$&$(2,2,4)$&$(3,1,4)$&$(4,4,0)$&$(2,1,3)$&$(4,1,1)$\\
\texttt{23}&$(1,3,3)$&$(3,4,1)$&$(3,2,2)$&$(3,1,4)$&$(3,0,4)$&$(1,3,1)$&$(2,1,0)$&$(2,0,4)$\\
\texttt{24}&$(1,3,3)$&$(2,1,4)$&$(3,2,2)$&$(2,4,1)$&$(3,1,1)$&$(0,3,0)$&$(3,2,3)$&$(1,1,1)$\\
\texttt{25}&$(1,3,4)$&$(0,1,0)$&$(0,2,1)$&$(2,3,1)$&$(2,0,4)$&$(4,3,3)$&$(4,4,0)$&$(3,4,2)$\\
\texttt{26}&$(1,3,4)$&$(0,4,0)$&$(0,2,1)$&$(3,2,4)$&$(0,1,3)$&$(4,1,4)$&$(2,3,0)$&$(0,1,0)$\\
\texttt{27}&$(1,4,3)$&$(3,2,1)$&$(4,4,4)$&$(2,0,2)$&$(0,1,3)$&$(1,4,1)$&$(2,3,0)$&$(0,4,0)$\\
\texttt{28}&$(1,4,3)$&$(2,3,4)$&$(4,4,4)$&$(3,0,3)$&$(2,0,4)$&$(1,2,2)$&$(4,4,0)$&$(2,1,3)$\\
\texttt{29}&$(1,4,4)$&$(2,4,4)$&$(4,3,4)$&$(1,1,4)$&$(0,4,0)$&$(1,0,2)$&$(3,4,3)$&$(3,2,2)$\\
\texttt{30}&$(1,4,4)$&$(3,1,1)$&$(4,3,4)$&$(4,4,1)$&$(2,4,2)$&$(2,2,0)$&$(1,1,3)$&$(2,4,4)$\\
\texttt{31}&$(2,0,2)$&$(2,1,2)$&$(2,3,3)$&$(3,3,3)$&$(2,3,4)$&$(1,2,3)$&$(2,2,3)$&$(1,0,2)$\\
\texttt{32}&$(2,0,2)$&$(3,4,3)$&$(2,3,3)$&$(2,2,2)$&$(4,0,3)$&$(1,3,3)$&$(4,2,3)$&$(4,2,3)$\\
\texttt{33}&$(2,0,4)$&$(4,3,3)$&$(4,4,0)$&$(3,4,2)$&$(1,3,4)$&$(0,1,0)$&$(0,2,1)$&$(2,3,1)$\\
\texttt{34}&$(2,0,4)$&$(1,2,2)$&$(4,4,0)$&$(2,1,3)$&$(1,4,3)$&$(2,3,4)$&$(4,4,4)$&$(3,0,3)$\\
\texttt{35}&$(2,1,0)$&$(3,0,1)$&$(3,0,4)$&$(4,2,4)$&$(1,0,3)$&$(3,4,4)$&$(3,3,3)$&$(2,0,4)$\\
\texttt{36}&$(2,1,0)$&$(2,0,4)$&$(3,0,4)$&$(1,3,1)$&$(3,2,2)$&$(3,1,4)$&$(1,3,3)$&$(3,4,1)$\\
\texttt{37}&$(2,1,3)$&$(1,4,4)$&$(3,1,4)$&$(1,1,0)$&$(4,1,2)$&$(4,3,3)$&$(2,1,4)$&$(0,3,1)$\\
\texttt{38}&$(2,1,3)$&$(4,1,1)$&$(3,1,4)$&$(4,4,0)$&$(2,2,4)$&$(2,2,4)$&$(1,2,2)$&$(3,3,2)$\\
\texttt{39}&$(2,1,4)$&$(0,3,1)$&$(4,1,2)$&$(4,3,3)$&$(3,1,4)$&$(1,1,0)$&$(2,1,3)$&$(1,4,4)$\\
\texttt{40}&$(2,1,4)$&$(0,2,4)$&$(4,1,2)$&$(1,2,2)$&$(0,1,1)$&$(3,0,1)$&$(4,4,3)$&$(4,3,3)$\\
\texttt{41}&$(2,2,3)$&$(1,0,2)$&$(2,3,4)$&$(1,2,3)$&$(2,3,3)$&$(3,3,3)$&$(2,0,2)$&$(2,1,2)$\\
\texttt{42}&$(2,2,3)$&$(4,0,3)$&$(2,3,4)$&$(4,3,2)$&$(3,4,1)$&$(1,0,2)$&$(2,4,0)$&$(0,1,0)$\\
\texttt{43}&$(2,2,4)$&$(3,3,1)$&$(1,2,2)$&$(2,2,3)$&$(4,4,3)$&$(1,2,2)$&$(0,1,1)$&$(2,0,4)$\\
\texttt{44}&$(2,2,4)$&$(2,2,4)$&$(1,2,2)$&$(3,3,2)$&$(2,1,3)$&$(4,1,1)$&$(3,1,4)$&$(4,4,0)$\\
\texttt{45}&$(2,3,0)$&$(0,1,0)$&$(0,1,3)$&$(4,1,4)$&$(0,2,1)$&$(3,2,4)$&$(1,3,4)$&$(0,4,0)$\\
\texttt{46}&$(2,3,0)$&$(0,4,0)$&$(0,1,3)$&$(1,4,1)$&$(4,4,4)$&$(2,0,2)$&$(1,4,3)$&$(3,2,1)$\\
\texttt{47}&$(2,3,3)$&$(3,3,3)$&$(2,0,2)$&$(2,1,2)$&$(2,2,3)$&$(1,0,2)$&$(2,3,4)$&$(1,2,3)$\\
\texttt{48}&$(2,3,3)$&$(2,2,2)$&$(2,0,2)$&$(3,4,3)$&$(4,2,3)$&$(4,2,3)$&$(4,0,3)$&$(1,3,3)$\\
\texttt{49}&$(2,3,4)$&$(1,2,3)$&$(2,2,3)$&$(1,0,2)$&$(2,0,2)$&$(2,1,2)$&$(2,3,3)$&$(3,3,3)$\\
\texttt{50}&$(2,3,4)$&$(4,3,2)$&$(2,2,3)$&$(4,0,3)$&$(2,4,0)$&$(0,1,0)$&$(3,4,1)$&$(1,0,2)$\\
\texttt{51}&$(2,4,0)$&$(0,1,0)$&$(3,4,1)$&$(1,0,2)$&$(2,3,4)$&$(4,3,2)$&$(2,2,3)$&$(4,0,3)$\\
\texttt{52}&$(2,4,0)$&$(0,4,0)$&$(3,4,1)$&$(4,0,3)$&$(4,0,3)$&$(4,2,2)$&$(4,2,3)$&$(1,3,2)$\\
\texttt{53}&$(2,4,2)$&$(2,2,0)$&$(1,1,3)$&$(2,4,4)$&$(1,4,4)$&$(3,1,1)$&$(4,3,4)$&$(4,4,1)$\\
\texttt{54}&$(2,4,2)$&$(3,3,0)$&$(1,1,3)$&$(3,1,1)$&$(3,3,2)$&$(4,4,3)$&$(3,4,2)$&$(0,4,3)$\\
\texttt{55}&$(3,0,2)$&$(2,4,4)$&$(3,2,1)$&$(0,2,0)$&$(1,1,2)$&$(1,0,1)$&$(4,2,2)$&$(0,3,0)$\\
\texttt{56}&$(3,0,2)$&$(3,1,1)$&$(3,2,1)$&$(0,3,0)$&$(0,3,0)$&$(4,1,2)$&$(4,1,3)$&$(1,4,2)$\\
\texttt{57}&$(3,0,4)$&$(4,2,4)$&$(2,1,0)$&$(3,0,1)$&$(3,3,3)$&$(2,0,4)$&$(1,0,3)$&$(3,4,4)$\\
\texttt{58}&$(3,0,4)$&$(1,3,1)$&$(2,1,0)$&$(2,0,4)$&$(1,3,3)$&$(3,4,1)$&$(3,2,2)$&$(3,1,4)$\\
\texttt{59}&$(3,1,1)$&$(0,2,0)$&$(3,2,3)$&$(4,4,4)$&$(3,3,3)$&$(3,0,1)$&$(1,0,3)$&$(2,1,1)$\\
\texttt{60}&$(3,1,1)$&$(0,3,0)$&$(3,2,3)$&$(1,1,1)$&$(1,3,3)$&$(2,1,4)$&$(3,2,2)$&$(2,4,1)$\\
\texttt{61}&$(3,1,4)$&$(1,1,0)$&$(2,1,3)$&$(1,4,4)$&$(2,1,4)$&$(0,3,1)$&$(4,1,2)$&$(4,3,3)$\\
\texttt{62}&$(3,1,4)$&$(4,4,0)$&$(2,1,3)$&$(4,1,1)$&$(1,2,2)$&$(3,3,2)$&$(2,2,4)$&$(2,2,4)$\\
\texttt{63}&$(3,2,1)$&$(0,2,0)$&$(3,0,2)$&$(2,4,4)$&$(4,2,2)$&$(0,3,0)$&$(1,1,2)$&$(1,0,1)$\\
\texttt{64}&$(3,2,1)$&$(0,3,0)$&$(3,0,2)$&$(3,1,1)$&$(4,1,3)$&$(1,4,2)$&$(0,3,0)$&$(4,1,2)$\\
\texttt{65}&$(3,2,2)$&$(2,4,1)$&$(1,3,3)$&$(2,1,4)$&$(3,2,3)$&$(1,1,1)$&$(3,1,1)$&$(0,3,0)$\\
\texttt{66}&$(3,2,2)$&$(3,1,4)$&$(1,3,3)$&$(3,4,1)$&$(2,1,0)$&$(2,0,4)$&$(3,0,4)$&$(1,3,1)$\\
\texttt{67}&$(3,2,3)$&$(4,4,4)$&$(3,1,1)$&$(0,2,0)$&$(1,0,3)$&$(2,1,1)$&$(3,3,3)$&$(3,0,1)$\\
\texttt{68}&$(3,2,3)$&$(1,1,1)$&$(3,1,1)$&$(0,3,0)$&$(3,2,2)$&$(2,4,1)$&$(1,3,3)$&$(2,1,4)$\\
\texttt{69}&$(3,3,2)$&$(4,4,3)$&$(3,4,2)$&$(0,4,3)$&$(2,4,2)$&$(3,3,0)$&$(1,1,3)$&$(3,1,1)$\\
\texttt{70}&$(3,3,2)$&$(1,1,2)$&$(3,4,2)$&$(0,1,2)$&$(0,4,0)$&$(4,0,3)$&$(3,4,3)$&$(2,3,3)$\\
\texttt{71}&$(3,3,3)$&$(3,0,1)$&$(1,0,3)$&$(2,1,1)$&$(3,1,1)$&$(0,2,0)$&$(3,2,3)$&$(4,4,4)$\\
\texttt{72}&$(3,3,3)$&$(2,0,4)$&$(1,0,3)$&$(3,4,4)$&$(3,0,4)$&$(4,2,4)$&$(2,1,0)$&$(3,0,1)$\\
\texttt{73}&$(3,4,1)$&$(1,0,2)$&$(2,4,0)$&$(0,1,0)$&$(2,2,3)$&$(4,0,3)$&$(2,3,4)$&$(4,3,2)$\\
\texttt{74}&$(3,4,1)$&$(4,0,3)$&$(2,4,0)$&$(0,4,0)$&$(4,2,3)$&$(1,3,2)$&$(4,0,3)$&$(4,2,2)$\\
\texttt{75}&$(3,4,2)$&$(0,1,2)$&$(3,3,2)$&$(1,1,2)$&$(3,4,3)$&$(2,3,3)$&$(0,4,0)$&$(4,0,3)$\\
\texttt{76}&$(3,4,2)$&$(0,4,3)$&$(3,3,2)$&$(4,4,3)$&$(1,1,3)$&$(3,1,1)$&$(2,4,2)$&$(3,3,0)$\\
\texttt{77}&$(3,4,3)$&$(3,2,2)$&$(0,4,0)$&$(1,0,2)$&$(4,3,4)$&$(1,1,4)$&$(1,4,4)$&$(2,4,4)$\\
\texttt{78}&$(3,4,3)$&$(2,3,3)$&$(0,4,0)$&$(4,0,3)$&$(3,4,2)$&$(0,1,2)$&$(3,3,2)$&$(1,1,2)$\\
\texttt{79}&$(4,0,3)$&$(4,2,2)$&$(4,2,3)$&$(1,3,2)$&$(2,4,0)$&$(0,4,0)$&$(3,4,1)$&$(4,0,3)$\\
\texttt{80}&$(4,0,3)$&$(1,3,3)$&$(4,2,3)$&$(4,2,3)$&$(2,0,2)$&$(3,4,3)$&$(2,3,3)$&$(2,2,2)$\\
\texttt{81}&$(4,1,2)$&$(4,3,3)$&$(2,1,4)$&$(0,3,1)$&$(2,1,3)$&$(1,4,4)$&$(3,1,4)$&$(1,1,0)$\\
\texttt{82}&$(4,1,2)$&$(1,2,2)$&$(2,1,4)$&$(0,2,4)$&$(4,4,3)$&$(4,3,3)$&$(0,1,1)$&$(3,0,1)$\\
\texttt{83}&$(4,1,3)$&$(1,4,2)$&$(0,3,0)$&$(4,1,2)$&$(3,2,1)$&$(0,3,0)$&$(3,0,2)$&$(3,1,1)$\\
\texttt{84}&$(4,1,3)$&$(4,1,3)$&$(0,3,0)$&$(1,4,3)$&$(1,1,1)$&$(4,2,1)$&$(0,4,3)$&$(3,2,3)$\\
\texttt{85}&$(4,2,2)$&$(0,2,0)$&$(1,1,2)$&$(4,0,4)$&$(1,1,1)$&$(1,3,4)$&$(0,4,3)$&$(2,3,2)$\\
\texttt{86}&$(4,2,2)$&$(0,3,0)$&$(1,1,2)$&$(1,0,1)$&$(3,2,1)$&$(0,2,0)$&$(3,0,2)$&$(2,4,4)$\\
\texttt{87}&$(4,2,3)$&$(1,3,2)$&$(4,0,3)$&$(4,2,2)$&$(3,4,1)$&$(4,0,3)$&$(2,4,0)$&$(0,4,0)$\\
\texttt{88}&$(4,2,3)$&$(4,2,3)$&$(4,0,3)$&$(1,3,3)$&$(2,3,3)$&$(2,2,2)$&$(2,0,2)$&$(3,4,3)$\\
\texttt{89}&$(4,3,4)$&$(4,4,1)$&$(1,4,4)$&$(3,1,1)$&$(1,1,3)$&$(2,4,4)$&$(2,4,2)$&$(2,2,0)$\\
\texttt{90}&$(4,3,4)$&$(1,1,4)$&$(1,4,4)$&$(2,4,4)$&$(3,4,3)$&$(3,2,2)$&$(0,4,0)$&$(1,0,2)$\\
\texttt{91}&$(4,4,0)$&$(3,4,2)$&$(2,0,4)$&$(4,3,3)$&$(0,2,1)$&$(2,3,1)$&$(1,3,4)$&$(0,1,0)$\\
\texttt{92}&$(4,4,0)$&$(2,1,3)$&$(2,0,4)$&$(1,2,2)$&$(4,4,4)$&$(3,0,3)$&$(1,4,3)$&$(2,3,4)$\\
\texttt{93}&$(4,4,3)$&$(4,3,3)$&$(0,1,1)$&$(3,0,1)$&$(4,1,2)$&$(1,2,2)$&$(2,1,4)$&$(0,2,4)$\\
\texttt{94}&$(4,4,3)$&$(1,2,2)$&$(0,1,1)$&$(2,0,4)$&$(2,2,4)$&$(3,3,1)$&$(1,2,2)$&$(2,2,3)$\\
\texttt{95}&$(4,4,4)$&$(3,0,3)$&$(1,4,3)$&$(2,3,4)$&$(4,4,0)$&$(2,1,3)$&$(2,0,4)$&$(1,2,2)$\\
\texttt{96}&$(4,4,4)$&$(2,0,2)$&$(1,4,3)$&$(3,2,1)$&$(2,3,0)$&$(0,4,0)$&$(0,1,3)$&$(1,4,1)$\\

\end{longtable}

\end{center}

\begin{figure}
\caption{Atkin-Lehner involutions on $\widetilde{X}_0(24)(\mathbb{F}_{5^3})$}
\[\begin{tikzpicture}\label{graph24mod125}
	\vertex[fill] (1) at (0,0) [label=left:$\texttt{1}$] {};
	\vertex[fill] (93) at (2,0) [label=right:$\texttt{93}$] {};
    \vertex (40) at (0,-2) [label=left:$\texttt{40}$] {};
	\vertex (82) at (2,-2) [label=right:$\texttt{82}$] {};
    \vertex (2) at (4,0) [label=left:$\texttt{2}$] {};
	\vertex (94) at (6,0) [label=right:$\texttt{94}$] {};
    \vertex[fill] (21) at (4,-2) [label=left:$\texttt{21}$] {};
	\vertex (43) at (6,-2) [label=right:$\texttt{43}$] {};
    \vertex[fill] (3) at (8,0) [label=left:$\texttt{3}$] {};
	\vertex[fill] (45) at (10,0) [label=right:$\texttt{45}$] {};
    \vertex (26) at (8,-2) [label=left:$\texttt{26}$] {};
	\vertex (5) at (10,-2) [label=right:$\texttt{5}$] {};
    \vertex (4) at (12,0) [label=left:$\texttt{4}$] {};
	\vertex (46) at (14,0) [label=right:$\texttt{46}$] {};
    \vertex (27) at (12,-2) [label=left:$\texttt{27}$] {};
	\vertex[fill] (96) at (14,-2) [label=right:$\texttt{96}$] {};
    \vertex (6) at (0,-3) [label=left:$\texttt{6}$] {};
	\vertex (25) at (2,-3) [label=right:$\texttt{25}$] {};
    \vertex[fill] (91) at (0,-5) [label=left:$\texttt{91}$] {};
	\vertex[fill] (33) at (2,-5) [label=right:$\texttt{33}$] {};
    \vertex[fill] (7) at (4,-3) [label=left:$\texttt{7}$] {};
	\vertex (83) at (6,-3) [label=right:$\texttt{83}$] {};
    \vertex (56) at (4,-5) [label=left:$\texttt{56}$] {};
	\vertex (64) at (6,-5) [label=right:$\texttt{64}$] {};
    \vertex[fill] (8) at (8,-3) [label=left:$\texttt{8}$] {};
	\vertex (84) at (10,-3) [label=right:$\texttt{84}$] {};
    \vertex (12) at (8,-5) [label=left:$\texttt{12}$] {};
	\vertex (16) at (10,-5) [label=right:$\texttt{16}$] {};
    \vertex[fill] (9) at (12,-3) [label=left:$\texttt{9}$] {};
	\vertex[fill] (77) at (14,-3) [label=right:$\texttt{77}$] {};
    \vertex (29) at (12,-5) [label=left:$\texttt{29}$] {};
	\vertex (90) at (14,-5) [label=right:$\texttt{90}$] {};
    \vertex (10) at (0,-6) [label=left:$\texttt{10}$] {};
	\vertex (78) at (2,-6) [label=right:$\texttt{78}$] {};
    \vertex (70) at (0,-8) [label=left:$\texttt{70}$] {};
	\vertex[fill] (75) at (2,-8) [label=right:$\texttt{75}$] {};
    \vertex[fill] (11) at (4,-6) [label=left:$\texttt{11}$] {};
	\vertex[fill] (15) at (6,-6) [label=right:$\texttt{15}$] {};
    \vertex (17) at (4,-8) [label=left:$\texttt{17}$] {};
	\vertex (85) at (6,-8) [label=right:$\texttt{85}$] {};
    \vertex (13) at (8,-6) [label=left:$\texttt{13}$] {};
	\vertex (72) at (10,-6) [label=right:$\texttt{72}$] {};
    \vertex[fill] (35) at (8,-8) [label=left:$\texttt{35}$] {};
	\vertex[fill] (57) at (10,-8) [label=right:$\texttt{57}$] {};
    \vertex (14) at (12,-6) [label=left:$\texttt{14}$] {};
	\vertex (71) at (14,-6) [label=right:$\texttt{71}$] {};
    \vertex[fill] (67) at (12,-8) [label=left:$\texttt{67}$] {};
	\vertex[fill] (59) at (14,-8) [label=right:$\texttt{59}$] {};
    \vertex (18) at (0,-9) [label=left:$\texttt{18}$] {};
	\vertex (86) at (2,-9) [label=right:$\texttt{86}$] {};
    \vertex[fill] (55) at (0,-11) [label=left:$\texttt{55}$] {};
	\vertex[fill] (63) at (2,-11) [label=right:$\texttt{63}$] {};
    \vertex[fill] (19) at (4,-9) [label=left:$\texttt{19}$] {};
	\vertex[fill] (53) at (6,-9) [label=right:$\texttt{53}$] {};
    \vertex (89) at (4,-11) [label=left:$\texttt{89}$] {};
	\vertex (30) at (6,-11) [label=right:$\texttt{30}$] {};
    \vertex (20) at (8,-9) [label=left:$\texttt{20}$] {};
	\vertex (54) at (10,-9) [label=right:$\texttt{54}$] {};
    \vertex[fill] (76) at (8,-11) [label=left:$\texttt{76}$] {};
	\vertex (69) at (10,-11) [label=right:$\texttt{69}$] {};
    \vertex[fill] (22) at (12,-9) [label=left:$\texttt{22}$] {};
	\vertex (44) at (14,-9) [label=right:$\texttt{44}$] {};
    \vertex (62) at (12,-11) [label=left:$\texttt{62}$] {};
	\vertex (38) at (14,-11) [label=right:$\texttt{38}$] {};
    \vertex (23) at (0,-12) [label=left:$\texttt{23}$] {};
	\vertex[fill] (66) at (2,-12) [label=right:$\texttt{66}$] {};
    \vertex (58) at (0,-14) [label=left:$\texttt{58}$] {};
	\vertex (36) at (2,-14) [label=right:$\texttt{36}$] {};
    \vertex (24) at (4,-12) [label=left:$\texttt{24}$] {};
	\vertex[fill] (65) at (6,-12) [label=right:$\texttt{65}$] {};
    \vertex (60) at (4,-14) [label=left:$\texttt{60}$] {};
	\vertex (68) at (6,-14) [label=right:$\texttt{68}$] {};
    \vertex (28) at (8,-12) [label=left:$\texttt{28}$] {};
	\vertex[fill] (95) at (10,-12) [label=right:$\texttt{95}$] {};
    \vertex (34) at (8,-14) [label=left:$\texttt{34}$] {};
	\vertex (92) at (10,-14) [label=right:$\texttt{92}$] {};
    \vertex[fill] (31) at (12,-12) [label=left:$\texttt{31}$] {};
	\vertex[fill] (47) at (14,-12) [label=right:$\texttt{47}$] {};
    \vertex (49) at (12,-14) [label=left:$\texttt{49}$] {};
	\vertex (41) at (14,-14) [label=right:$\texttt{41}$] {};
    \vertex (32) at (0,-15) [label=left:$\texttt{32}$] {};
	\vertex (48) at (2,-15) [label=right:$\texttt{48}$] {};
    \vertex[fill] (80) at (0,-17) [label=left:$\texttt{80}$] {};
	\vertex (88) at (2,-17) [label=right:$\texttt{88}$] {};
    \vertex[fill] (37) at (4,-15) [label=left:$\texttt{37}$] {};
	\vertex[fill] (61) at (6,-15) [label=right:$\texttt{61}$] {};
    \vertex (81) at (4,-17) [label=left:$\texttt{81}$] {};
	\vertex (39) at (6,-17) [label=right:$\texttt{39}$] {};
    \vertex (42) at (8,-15) [label=left:$\texttt{42}$] {};
	\vertex (50) at (10,-15) [label=right:$\texttt{50}$] {};
    \vertex[fill] (73) at (8,-17) [label=left:$\texttt{73}$] {};
	\vertex[fill] (51) at (10,-17) [label=right:$\texttt{51}$] {};
    \vertex (52) at (12,-15) [label=left:$\texttt{52}$] {};
	\vertex (74) at (14,-15) [label=right:$\texttt{74}$] {};
    \vertex[fill] (79) at (12,-17) [label=left:$\texttt{79}$] {};
	\vertex (87) at (14,-17) [label=right:$\texttt{87}$] {};

\path
		(1) edge (93)    (1) edge (40)    (1) edge (82)
        (2) edge (94)    (2) edge (21)    (2) edge (43)
        (3) edge (45)    (3) edge (26)    (3) edge (5)
        (4) edge (46)    (4) edge (27)    (4) edge (96)
        (5) edge (26)    (5) edge (45)    (5) edge (3)
        (6) edge (25)    (6) edge (91)    (6) edge (33)
        (7) edge (83)    (7) edge (56)    (7) edge (64)
        (8) edge (84)    (8) edge (12)    (8) edge (16)
        (9) edge (77)    (9) edge (29)    (9) edge (90)
        (10) edge (78)    (10) edge (70)    (10) edge (75)
        (11) edge (15)    (11) edge (17)    (11) edge (85)
        (12) edge (16)    (12) edge (8)    (12) edge (84)
        (13) edge (72)    (13) edge (35)    (13) edge (57)
        (14) edge (71)    (14) edge (67)    (14) edge (59)
        (15) edge (11)    (15) edge (85)    (15) edge (17)
        (16) edge (12)    (16) edge (84)    (16) edge (8)
        (17) edge (85)    (17) edge (11)    (17) edge (15)
        (18) edge (86)    (18) edge (55)    (18) edge (63)
        (19) edge (53)    (19) edge (89)    (19) edge (30)
        (20) edge (54)    (20) edge (76)    (20) edge (69)
        (21) edge (43)    (21) edge (2)    (21) edge (94)
        (22) edge (44)    (22) edge (62)    (22) edge (38)
        (23) edge (66)    (23) edge (58)    (23) edge (36)
        (24) edge (65)    (24) edge (60)    (24) edge (68)
        (25) edge (6)    (25) edge (33)    (25) edge (91)
        (26) edge (5)    (26) edge (3)    (26) edge (45)
        (27) edge (96)    (27) edge (4)    (27) edge (46)
        (28) edge (95)    (28) edge (34)    (28) edge (92)
        (29) edge (90)    (29) edge (9)    (29) edge (77)
        (30) edge (89)    (30) edge (53)    (30) edge (19)
        (31) edge (47)    (31) edge (49)    (31) edge (41)
        (32) edge (48)    (32) edge (80)    (32) edge (88)
        (33) edge (91)    (33) edge (25)    (33) edge (6)
        (34) edge (92)    (34) edge (28)    (34) edge (95)
        (35) edge (57)    (35) edge (13)    (35) edge (72)
        (36) edge (58)    (36) edge (66)    (36) edge (23)
        (37) edge (61)    (37) edge (81)    (37) edge (39)
        (38) edge (62)    (38) edge (44)    (38) edge (22)
        (39) edge (81)    (39) edge (61)    (39) edge (37)
        (40) edge (82)    (40) edge (1)    (40) edge (93)
        (41) edge (49)    (41) edge (47)    (41) edge (31)
        (42) edge (50)    (42) edge (73)    (42) edge (51)
        (43) edge (21)    (43) edge (94)    (43) edge (2)
        (44) edge (22)    (44) edge (38)    (44) edge (62)
        (45) edge (3)    (45) edge (5)    (45) edge (26)
        (46) edge (4)    (46) edge (96)    (46) edge (27)
        (47) edge (31)    (47) edge (41)    (47) edge (49)
        (48) edge (32)    (48) edge (88)    (48) edge (80)
        (49) edge (41)    (49) edge (31)    (49) edge (47)
        (50) edge (42)    (50) edge (51)    (50) edge (73)
        (51) edge (73)    (51) edge (50)    (51) edge (42)
        (52) edge (74)    (52) edge (79)    (52) edge (87)
        (53) edge (19)    (53) edge (30)    (53) edge (89)
        (54) edge (20)    (54) edge (69)    (54) edge (76)
        (55) edge (63)    (55) edge (18)    (55) edge (86)
        (56) edge (64)    (56) edge (7)    (56) edge (83)
        (57) edge (35)    (57) edge (72)    (57) edge (13)
        (58) edge (36)    (58) edge (23)    (58) edge (66)
        (59) edge (67)    (59) edge (71)    (59) edge (14)
        (60) edge (68)    (60) edge (24)    (60) edge (65)
        (61) edge (37)    (61) edge (39)    (61) edge (81)
        (62) edge (38)    (62) edge (22)    (62) edge (44)
        (63) edge (55)    (63) edge (86)    (63) edge (18)
        (64) edge (56)    (64) edge (83)    (64) edge (7)
        (65) edge (24)    (65) edge (68)    (65) edge (60)
        (66) edge (23)    (66) edge (36)    (66) edge (58)
        (67) edge (59)    (67) edge (14)    (67) edge (71)
        (68) edge (60)    (68) edge (65)    (68) edge (24)
        (69) edge (76)    (69) edge (54)    (69) edge (20)
        (70) edge (75)    (70) edge (10)    (70) edge (78)
        (71) edge (14)    (71) edge (59)    (71) edge (67)
        (72) edge (13)    (72) edge (57)    (72) edge (35)
        (73) edge (51)    (73) edge (42)    (73) edge (50)
        (74) edge (52)    (74) edge (87)    (74) edge (79)
        (75) edge (70)    (75) edge (78)    (75) edge (10)
        (76) edge (69)    (76) edge (20)    (76) edge (54)
        (77) edge (9)    (77) edge (90)    (77) edge (29)
        (78) edge (10)    (78) edge (75)    (78) edge (70)
        (79) edge (87)    (79) edge (52)    (79) edge (74)
        (80) edge (88)    (80) edge (32)    (80) edge (48)
        (81) edge (39)    (81) edge (37)    (81) edge (61)
        (82) edge (40)    (82) edge (93)    (82) edge (1)
        (83) edge (7)    (83) edge (64)    (83) edge (56)
        (84) edge (8)    (84) edge (16)    (84) edge (12)
        (85) edge (17)    (85) edge (15)    (85) edge (11)
        (86) edge (18)    (86) edge (63)    (86) edge (55)
        (87) edge (79)    (87) edge (74)    (87) edge (52)
        (88) edge (80)    (88) edge (48)    (88) edge (32)
        (89) edge (30)    (89) edge (19)    (89) edge (53)
        (90) edge (29)    (90) edge (77)    (90) edge (9)
        (91) edge (33)    (91) edge (6)    (91) edge (25)
        (92) edge (34)    (92) edge (95)    (92) edge (28)
        (93) edge (1)    (93) edge (82)    (93) edge (40)
        (94) edge (2)    (94) edge (43)    (94) edge (21)
        (95) edge (28)    (95) edge (92)    (95) edge (34)
        (96) edge (27)    (96) edge (46)    (96) edge (4)
	;
\end{tikzpicture}\]
\end{figure}
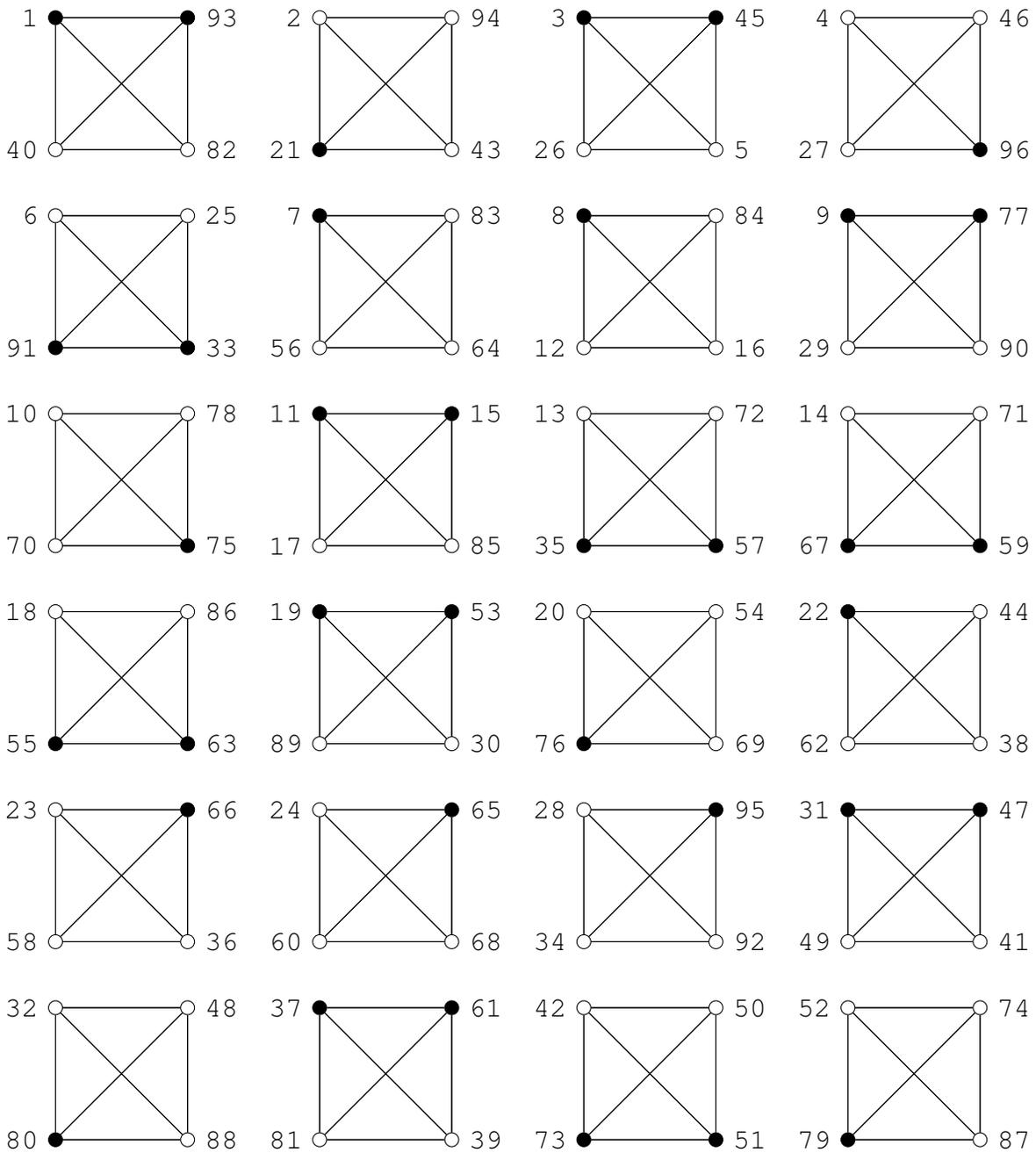

The information in Table \ref{table24-125involution} can be represented in Figure \ref{graph24mod125}. The $96$-vertex graph has $24$ connected components, which are all complete graphs of order 4. The black dot points are those such that $$\varphi(P)\in\{(-3,0),(-4,2),(0,6)\}$$ while the white dot points are those such that $$\varphi(P)\not\in\{(-3,0),(-4,2),(0,6)\}$$
The horizontal lines represent $\omega_3$. The vertical lines represent $\omega_8$. And the diagonal lines represent $\omega_{24}$. It is clear that each white dot point is connected to  at least one black dot point, which is not allowed by Lemma \ref{Frobenius}.

\end{proof}

\section*{}
\subsection*{Acknowledgements}
I appreciate Sheldon Kamienny for providing many valuable ideas and insightful comments.

\end{document}